\documentclass[a4paper,12pt]{amsart}
\usepackage{amssymb}
\usepackage{amsmath}
\usepackage{ifthen}
\usepackage{graphicx,float}
\usepackage{mathtools}
\usepackage{latexsym, amsthm, amscd, euscript, float, subfig}
\newtheorem{theorem}{Theorem}[section]
\newtheorem{lemma}[theorem]{Lemma}
\newtheorem{corollary}[theorem]{Corollary}

\theoremstyle{definition}

\newtheorem{example}[theorem]{Example}
\newtheorem{remark}[theorem]{Remark}

\DeclarePairedDelimiter\ceil{\lceil}{\rceil}

\newcounter{minutes}
\divide\time by 60
\newcounter{hours}
\multiply\time by 60
\addtocounter{minutes}{-\time}

\numberwithin{equation}{section}
\begin{document}

\bibliographystyle{amsplain}

\title[Properties of $\beta$-Ces\`{a}ro operators]
{Properties of $\beta$-Ces\`{a}ro operators on $\alpha$-Bloch space }

\def\thefootnote{}
\footnotetext{ \texttt{\tiny File:~\jobname .tex,
          printed: \number\day-\number\month-\number\year,
          \thehours.\ifnum\theminutes<10{0}\fi\theminutes}
} \makeatletter\def\thefootnote{\@arabic\c@footnote}\makeatother

\author[Shankey Kumar]{Shankey Kumar}
\address{Shankey Kumar, Discipline of Mathematics,
Indian Institute of Technology Indore,
Indore 453552, India.}
\email{shankeygarg93@gmail.com}

\author[S. K. Sahoo]{Swadesh Kumar Sahoo$^*$}
\address{Swadesh Kumar Sahoo, Discipline of Mathematics,
Indian Institute of Technology Indore,
Indore 453552, India.}
\email{swadesh.sahoo@iiti.ac.in}

\keywords{The $\alpha$-Bloch space, The $\beta$-Ces\`aro operator, Norms, Essential norm, Spectrum, Compact operator, 
Separable space, Bounded operator\\
${}^{\mathbf{*}}$ {\tt Corresponding author}}

\begin{abstract}
For each $ \alpha > 0 $, the $\alpha$-Bloch space
is consisting of all analytic functions $f$ on the unit disk satisfying
$ \sup_{|z|<1} (1-|z|^2)^\alpha |f'(z)| < + \infty.$
In this paper, we consider the following complex integral operator, namely the $\beta$-Ces\`{a}ro operator
\begin{equation}
 C_\beta(f)(z)=\int_{0}^{z}\frac{f(w)}{w(1-w)^{\beta}}dw \nonumber
\end{equation}
and its generalization, acting from the $\alpha$-Bloch space to itself, where $f(0)=0$ and $\beta\in\mathbb{R}$.
We investigate the boundedness and compactness of the $\beta$-Ces\`{a}ro operators and their generalization.
Also we calculate the essential norm and spectrum of these operators.
\end{abstract}

\maketitle



\section{\textbf{Introduction}}\label{1sec1}
Let $\mathcal{H}_0$ denote the class of analytic functions $f$ in the unit disk 
$\mathbb{D}:=\{z\in\mathbb{C}:|z|<1\}$ with $f(0)=0$.
 In this paper, we consider a complex integral operator, called {\em $\beta$-Ces\`{a}ro operator}, denoted by $C_\beta$, for $\beta \in \mathbb{R}$,
is defined on the space $\mathcal{H}_0$  as  
\begin{equation}\label{1eq1.1}
C_\beta(f)(z)=\int_{0}^{z}\frac{f(w)}{w(1-w)^{\beta}}dw. 
\end{equation}

This operator includes the Alexander operator ($\beta=0$), see \cite[chapter 8]{Duren83} and \cite[chapter 1]{Miller90}, and the Ces\`{a}ro operator ($\beta=1$), see  \cite{Hartmann74}.
We can further generalize this operator, if we replace $(1-w)^{-\beta}$ by 
\begin{equation}\label{1eq1.2}
g_{\beta}(w)=\sum_{j=1}^{k}\frac{a_j}{{(1-b_{j} w)}^{\beta}}+h(w)
\end{equation}
in \eqref{1eq1.1}, where $b_j$, $1\leq j\leq k$, are distinct points on the unit circle, 
$|a_j|>0,\,\ \forall \ j$, and $h$ is  bounded analytic function in $\mathbb{D}$. We call this operator as
{\em generalized $\beta$-Ces\`{a}ro operator} because for $b_j=a_j=k=1$ and $h=0$, we obtain the $\beta$-Ces\`{a}ro operator. The generalized $\beta$-Ces\`{a}ro operator denoted by 
 $C_{g_\beta}$. For the choices $\beta=0$ and $\beta=1$ in \eqref{1eq1.2}, 
 the generalized $\beta$- Ces\`{a}ro operators are respectively called as the {\em generalized Alexander operator} and the {\em generalized Ces\`{a}ro operator}.

The Alexander operator and the Ces\`{a}ro operator are used by many authors for different purposes, see \cite{Duren83,Hartmann74,Miller90,Ponnusamy18}.
Moreover, boundedness of the Ces\`{a}ro and related operators in various function spaces
are studied in the literature; see \cite{DS93,Mia92,Sis90,Xia97}.
In this paper we study these operators as linear operators on $\alpha$-Bloch space, denoted by $\mathcal{B_\alpha}$,
and is defined as follows: the {\em $\alpha$-Bloch space} \cite{Zhu93} of $\mathbb{D}$, denoted by $\mathcal{B_\alpha} $, for each $ \alpha > 0 $,
is consisting of analytic functions $f$ on $\mathbb{D}$ such that
$$
  \sup \limits_{z\in\mathbb{D}} (1-|z|^2)^\alpha |f'(z)| < + \infty.
$$
 The space $\mathcal{B_\alpha}$ is a complex Banach space with the norm 
\begin{equation}\label{1eq1.4}
 \|f\|=|f(0)|+\|f\|_{\mathcal{B_\alpha}},
 \end{equation}
 whereas $\|f\|_{\mathcal{B_\alpha}}=\sup_{z\in\mathbb{D}} (1-|z|^2)^\alpha |f'(z)|$ represents a semi-norm. The proof of this follows from the proof of Proposition 2.5 of \cite{Hidetaka17}.
 If we restrict this space with the condition $f(0)=0$, for $f \in \mathcal{B_\alpha}$, this restricted space is a subspace 
 of $\mathcal{B_\alpha}$, denoted by $\mathcal{B}_\alpha^0$.
 The semi-norm $\|.\|_{\mathcal{B_\alpha}}$ on $\mathcal{B_\alpha}$ becomes norm on $\mathcal{B}_\alpha^0$.
 We observe that $\mathcal{B}_\alpha^0$ is a Banach space with norm $\|.\|_{\mathcal{B_\alpha}}$ and proof of this is explained in Section~\ref{1sec2}.
Throughout this paper we consider $\alpha>0$ unless it is specified.
More on literature survey about the $1$-Bloch space can be found in \cite{Zhu05,Zhu07}.

  Many authors study integral operators on analytic function spaces. For instance, Stevic studies compactness and essential norm
  of the integral type operator 
  \begin{equation}
    P_\varphi^g(f)(z)=\int_{0}^{1}f(\varphi(tz))g(tz)\frac{dt}{t}\nonumber
  \end{equation}
where $g$ is an analytic function in $\mathbb{D}$, $g(0)=0$ and $\varphi$ is a holomorphic self-map of $\mathbb{D}$, 
acting on Bloch-type spaces, see \cite{Stevic09,Stevic10}. 
Secondly, in \cite{AHLNP05}, boundedness of generalized Ces\`{a}ro averaging operators 
on certain function spaces are investigated. 
These operators are very similar 
to our operator $C_{g_\beta}$, but they do not simultaneously include the Alexander operator as well as the Ces\`{a}ro operator. Main motive of this 
paper is to study spectral properties of generalized $\beta$-Ces\`{a}ro operators on $\mathcal{B}_\alpha^0$ which simultaneously include the Alexander operator as well as the Ces\`{a}ro operator.

In this scenario, the second section contains boundedness property of the $\beta$-Ces\`{a}ro operators
and also we have some examples to explain the unboundedness property of the $\beta$-Ces\`{a}ro operators. Compactness of these operators are studied in Section~\ref{1sec3}. In particular,
essential norm and spectrum are calculated in Section~\ref{1sec4} and \ref{1sec5}. Finally, we include an application section which assures that $\mathcal{B}_\alpha^0$ is a separable space, for each $\alpha>0$.


\section{\textbf{ Boundedness of the $\beta$-Ces\`{a}ro operators}}\label{1sec2}
 In this section, we discuss the boundedness and unboundedness of the $\beta$-Ces\`{a}ro operators, defined by \eqref{1eq1.1}, on $\mathcal{B}_\alpha^0$.
 At the end of this section, we provide some illustrative examples to show that the $\beta$-Ces\`{a}ro operators are unbounded linear operators
 on $\mathcal{B}_\alpha^0$, for some $\beta$.
In Table~\ref{table}, we discuss all restrictions on $\beta$ for which the $\beta$-Ces\`{a}ro operators are bounded and unbounded. In the sequel, first we describe the completeness property of $\mathcal{B}_\alpha^0$ under $\|.\|_{\mathcal{B_\alpha}}$.
 
\begin{table}[H]
\begin{tabular}{|c|c|c|c|} 
\hline
Bounded & Unbounded\\
\hline
$\beta \leq \alpha <1$ & $\beta>\alpha$ (Example \ref{1ex2.6}) \\ 
$\alpha>1\geq\beta$ & $\alpha=\beta \geq 1$ (Example \ref{1ex2.7}) \\ 
$\beta<\alpha=1$  & $\alpha>\beta>1$ (Example \ref{1ex2.8})   \\ 
\hline
\end{tabular}
\vspace*{0.2cm}
\caption{Boundedness of $\beta$-Ces\`{a}ro operator} \label{table}
\end{table}

\begin{theorem}
  For each $\alpha >0 $,  $(\mathcal{B}_\alpha^0, \|.\|_{\mathcal{B_\alpha}} )$ is a Banach space.
 \end{theorem}
\begin{proof}
  We know by \cite[Proposition~1]{Zhu93} that for each $\alpha >0$, $\mathcal{B_\alpha}$ is a Banach space. 
The only thing we need to show that this subspace is a closed subspace of $\mathcal{B_\alpha}$. 
Let ${(f_{n})}_{n \in \mathbb{N}}$ be a Cauchy sequence in ($\mathcal{B}_\alpha^0, \|.\|_{\mathcal{B_\alpha}} $). 
Then for $\epsilon > 0$, there exists a $P \in \mathbb{N}$ such that 
$$
 \|f_n-f_m\|_{\mathcal{B_\alpha}} < \epsilon \quad \mbox{for all $n,m\geq P$}.
$$
As we know that
$$
 \|g\|_{\mathcal{B_\alpha}} = \|g\| , \quad \mbox{for all $g \in \mathcal{B}_\alpha^0$}   
$$
we obtain
$$
 \|f_n-f_m\|< \epsilon ,\hspace{5mm} \forall n,m\geq P. 
$$
   Then, there exist a function $ f \in  \mathcal{B}_\alpha$ such that 
 ${(f_{n})}_{n \in \mathbb{N}}$ converges to $f$. Now, we need to show that $f \in \mathcal{B}_\alpha^0 $. By using equation \eqref{1eq1.4}, we have the following property
$$
   \|.\|_{\mathcal{B_\alpha}} \leq \|.\|.
 $$
This property imply that the sequence ($\|f_n\|_{\mathcal{B_\alpha}}$) converges to $\|f\|_{\mathcal{B_\alpha}}$, equivalently, from here, we can say that 
 the sequence ($\|f_n\|$) converges to $\|f\|_{\mathcal{B_\alpha}}$, and also, this sequence converges to $\|f\|$.
 From the uniqueness of the limit of convergent sequence,
 $$
 \|f\| = \|f\|_{\mathcal{B_\alpha}}.
$$
Since $f(0)=0$, consequently we have $f \in \mathcal{B}_\alpha^0$.
\end{proof}
  
To obtain our desired results, we need the following lemma. The proof of this lemma plays a key role
in most of the proofs of our main results. Therefore, we discuss the proof of this lemma in this section.
\begin{lemma} \cite{Julio10}\label{1lemma2.1}
For $\alpha>0$, let $f \in \mathcal{B}_\alpha$, we have the following basic properties:
\begin{enumerate}
 \item[\bf (i)] If $\alpha<1$, then f is a bounded analytic function.
 \item[\bf (ii)]  If $\alpha=1$, then
$$
|f(z)|\leq |f(0)|+\frac{\|f\|_{\mathcal{B}_1}}{2}\log \bigg(\frac{1+|z|}{1-|z|}\bigg).
$$
\item[\bf (iii)] If $\alpha>1$, then
$$
|f(z)|\leq |f(0)|+\frac{\|f\|_{\mathcal{B}_\alpha}}{\alpha-1} \bigg(\frac{1}{(1-|z|)^{\alpha-1}}-1\bigg).
$$
\end{enumerate}
\end{lemma}

\begin{proof} Suppose $f \in \mathcal{B}_\alpha$ and $z \in \mathbb{D}$, then
\begin{align}
 |f(z)-f(0)| &= \bigg|z\int_0^1 f'(zt)dt\bigg| \leq |z|\int_0^1 |f'(zt)|dt. \nonumber
\end{align}
By using the definition of $\alpha$-Bloch space, we have
\begin{align}
\label{1eq2.1}  |f(z)-f(0)| &\leq |z|\|f\|_{\mathcal{B}_\alpha} \int_0^1 \frac{1}{(1-|z|^2 t^2)^{\alpha}}dt.
\end{align}
Since $(1+|z|t)\geq 1$, we obtain
\begin{align}                         
 \label{1eq2.2} |f(z)-f(0)| &\leq |z|\|f\|_{\mathcal{B}_\alpha} \int_0^1 \frac{1}{(1-|z|t)^{\alpha}}dt \nonumber\\
 &\leq \|f\|_{\mathcal{B}_\alpha}\frac{1}{1-\alpha} \bigg(1-\frac{1}{(1-|z|)^{\alpha-1}}\bigg).
\end{align}
We now complete the proofs of (i)-(iii) as described below.

\subsection*{(i)} 
We notice that
\begin{align}
  |f(z)|&\leq|f(0)|+ \|f\|_{\mathcal{B}_\alpha}\frac{1}{1-\alpha} \big(1-(1-|z|)^{1-\alpha}\big).  \nonumber
\end{align}
Since $1-(1-|z|)^{1-\alpha}\leq1$, we obtain
\begin{align}
  \label{1eq2.3} |f(z)|&\leq|f(0)|+ \|f\|_{\mathcal{B}_\alpha}\frac{1}{1-\alpha}. 
\end{align}
\subsection*{(ii)} From \eqref{1eq2.1}, we estimate 
\begin{align}\label{1eq2.4}
 |f(z)|&\leq |f(0)|+\frac{\|f\|_{\mathcal{B}_1}}{2}\log \bigg(\frac{1+|z|}{1-|z|}\bigg).
\end{align}
\subsection*{(iii)} It easily follows that
\begin{align}
 |f(z)-f(0)|&\leq \|f\|_{\mathcal{B}_\alpha}\frac{1}{\alpha-1} \bigg(\frac{1}{(1-|z|)^{\alpha-1}}-1\bigg). \nonumber
\end{align}
By using triangle inequality, we finally obtain
\begin{align}\label{1eq2.5}
  |f(z)|&\leq |f(0)|+\frac{\|f\|_{\mathcal{B}_\alpha}}{\alpha-1} \bigg(\frac{1}{(1-|z|)^{\alpha-1}}-1\bigg).
\end{align}
This completes the proof of our lemma.
\end{proof}

For $f \in \mathcal{B}_\alpha^0$, $C_\beta(f)$ is an analytic function in $\mathbb{D}$ 
and $C_\beta{(f)(0)}=0$. Now, we have three consecutive theorems, which describe the boundedness of 
$\beta$-Ces\`{a}ro operators from $\mathcal{B}_\alpha^0$ to $\mathcal{B}_\alpha^0$
for three different restrictions on $\beta$.

\begin{theorem}\label{1theorem2.2}
The $\beta$-Ces\`{a}ro operator is a bounded linear operator from $\mathcal{B}_\alpha^0$ to $\mathcal{B}_\alpha^0$, for $\beta \leq \alpha <1$. 
\end{theorem}
\begin{proof}
Suppose that $f \in \mathcal{B}_\alpha^0$, for $\alpha < 1$. From \eqref{1eq2.2}, we have 
\begin{align*}
 (1-|z|^2)^{\alpha} \bigg|\frac{f(z)}{z(1-z)^{\beta}}\bigg| &\leq \frac{(1+|z|)^{\alpha}}{|z|}\frac{\|f\|_{\mathcal{B}_\alpha}}{(1-\alpha)} \big[(1-|z|)^{\alpha-\beta}-(1-|z|)^{1-\beta}\big].
\end{align*}
For $\beta \leq \alpha <1$, this leads to
\begin{align*}
 (1-|z|^2)^{\alpha} \bigg|\frac{f(z)}{z(1-z)^{\beta}}\bigg| &\leq \frac{(1+|z|)^{\alpha}}{|z|}\frac{\|f\|_{\mathcal{B}_\alpha}}{(1-\alpha)}(1-|z|)^{\alpha-\beta} \big[1-(1-|z|)^{1-\alpha}\big]. 
\end{align*}
Now for $\alpha < 1$, we have $1-(1-|z|)^{1-\alpha} \leq |z|$ and $z$ is arbitrary point here, therefore
$$
  \|C_\beta(f)\|_{\mathcal{B}_\alpha} \leq \sup\bigg \lbrace\frac{(1+|z|)^\alpha (1-|z|)^{\alpha-\beta}}{(1-\alpha)} :z\in \mathbb{D}\bigg \rbrace\|f\|_{\mathcal{B}_\alpha}.
$$
This concludes the proof.
\end{proof}

\begin{theorem}\label{1theorem2.3}
The $\beta$-Ces\`{a}ro operator is a bounded linear operator from $\mathcal{B}_\alpha^0$ to $\mathcal{B}_\alpha^0$, for  $\beta \leq 1<\alpha$. 
\end{theorem}
\begin{proof}
Suppose that $f \in \mathcal{B}_\alpha^0$, for $\alpha > 1$. From \eqref{1eq2.5}, we have 
\begin{align*}
 (1-|z|^2)^{\alpha} \bigg|\frac{f(z)}{z(1-z)^{\beta}}\bigg| &\leq \frac{(1+|z|)^{\alpha}}{|z|}\frac{\|f\|_{\mathcal{B}_\alpha}}{(\alpha-1)} \big[(1-|z|)^{1-\beta}-(1-|z|)^{\alpha-\beta}\big].
\end{align*}
If $\beta \leq 1<\alpha$, it leads to
\begin{align*}
 (1-|z|^2)^{\alpha} \bigg|\frac{f(z)}{z(1-z)^{\beta}}\bigg| &\leq \frac{(1+|z|)^{\alpha}}{|z|}\frac{\|f\|_{\mathcal{B}_\alpha}}{(\alpha-1)} (1-|z|)^{1-\beta}\big[1-(1-|z|)^{\alpha-1}\big].
 \end{align*}
Since $1-(1-|z|)^{\alpha-1} \leq 1-(1-|z|)^{\ceil{\alpha}}$, where $\ceil{.}$ is a Greatest Integer Function, we obtain
\begin{align*}
 (1-|z|^2)^{\alpha} \bigg|\frac{f(z)}{z(1-z)^{\beta}}\bigg| &\leq \frac{(1+|z|)^{\alpha}}{|z|}\frac{\|f\|_{\mathcal{B}_\alpha}}{(\alpha-1)} (1-|z|)^{1-\beta}\big[1-(1-|z|)^{\ceil{\alpha}}\big].
 \end{align*}
 Note that
 $$
  (1-|z|)^{\ceil{\alpha}}=\sum_{k=0}^{\ceil{\alpha}}{{\ceil{\alpha}}\choose{k}}(-|z|)^k.
 $$
Thus, we obtain
\begin{align*}
 (1-|z|^2)^{\alpha} \bigg|\frac{f(z)}{z(1-z)^{\beta}}\bigg| &\leq \frac{(1+|z|)^\alpha (1-|z|)^{1-\beta}}{(\alpha-1)} \sum_{k=1}^{\ceil{\alpha}} {{\ceil{\alpha}}\choose{k}}(-|z|)^{k-1}.
 \end{align*}
 Since $z$ is arbitrary point in $\mathbb{D}$, therefore we have
\begin{align*}
& \|C_\beta(f)\|_{\mathcal{B}_\alpha}\\
& \hspace{0.5cm}\leq \sup \bigg \lbrace\frac{(1+|z|)^\alpha (1-|z|)^{1-\beta}}{(\alpha-1)} 
\sum_{k=1}^{\ceil{\alpha}} {{{\ceil{\alpha}}}\choose{k}}(-|z|)^{k-1}:z\in \mathbb{D}\bigg \rbrace\|f\|_{\mathcal{B}_\alpha},
\end{align*}
which concludes the proof.
\end{proof}

For $\beta=1$, the $\beta$-Ces\`{a}ro operator is nothing but the Ces\`{a}ro operator. Thus, 
Theorem \ref{1theorem2.3} yields the following bounded property of the Ces\`{a}ro operator.
\begin{corollary}
 The Ces\`{a}ro operator is a bounded linear operator from $\mathcal{B}_\alpha^0$ to $\mathcal{B}_\alpha^0$, for $\alpha > 1$.
\end{corollary}

\begin{theorem}\label{1theorem2.4}
The $\beta$-Ces\`{a}ro operator is a bounded linear operator from $\mathcal{B}_\alpha^0$ to $\mathcal{B}_\alpha^0$, for $\beta<\alpha=1$. 
\end{theorem}
\begin{proof}
Suppose that $f \in \mathcal{B}_\alpha^0$, for $\alpha=1$. From \eqref{1eq2.4}, we obtain 
\begin{align*}
 (1-|z|^2) \bigg|\frac{f(z)}{z(1-z)^{\beta}}\bigg| &\leq \frac{(1-|z|^2)}{|z|(1-|z|)^\beta}\frac{\|f\|_{\mathcal{B}_1}}{2}\log \bigg(\frac{1+|z|}{1-|z|}\bigg).
\end{align*}
For $\beta<\alpha=1$, we have
\begin{align*}
 (1-|z|^2) \bigg|\frac{f(z)}{z(1-z)^{\beta}}\bigg| &\leq \frac{(1-|z|^2)}{2|z|(1-|z|)^\beta}\log \bigg(\frac{1+|z|}{1-|z|}\bigg)\|f\|_{\mathcal{B}_1}.
\end{align*}
Since $z$ is arbitrary point, we obtain that
$$
  \|C_\beta(f)\|_{\mathcal{B}_1} \leq \sup\bigg \lbrace \frac{(1-|z|)^{1-\beta}}{|z|}\log \bigg(\frac{1+|z|}{1-|z|}\bigg):z\in \mathbb{D}\bigg \rbrace\|f\|_{\mathcal{B}_1},
$$
completing the proof.
\end{proof}
As an immediate consequence of Theorems \ref{1theorem2.2}, \ref{1theorem2.3} and \ref{1theorem2.4}, we can easily prove 
the following corollary.
\begin{corollary}\label{1corollary2.2}
The generalized $\beta$-Ces\`{a}ro operator is a bounded linear operator from $\mathcal{B}_\alpha^0$ to $\mathcal{B}_\alpha^0$, either  for $\beta \leq \alpha<1$ or $\beta \leq 1<\alpha$ or $\beta<\alpha=1$. 
\end{corollary}
\begin{proof} Let $f \in \mathcal{B}_\alpha^0$ and consider the generalized $\beta$-Ces\`{a}ro operator, either for $\beta \leq \alpha<1$ or $\beta \leq 1<\alpha$ or $\beta<\alpha=1$.
Then from \eqref{1eq1.2}, we obtain
 \begin{align*}
(1-|z|^2)^\alpha \bigg|\frac{f(z)g_{\beta}(z)}{z}\bigg|&=(1-|z|^2)^\alpha\bigg|\frac{f(z)}{z} \bigg(\sum_{j=1}^{k}\frac{a_j}{{(1-b_{j} z)}^{\beta}}+h(z)\bigg) \bigg|.
 \end{align*}
 By triangle inequality
 \begin{align*}
(1-|z|^2)^\alpha \bigg|\frac{f(z)g_{\beta}(z)}{z}\bigg|\leq& \sum_{j=1}^{k} |a_k| (1-|z|^2)^\alpha \frac{|f(z)|}{|z|(1-|z|)^\beta} 
\\&+ \|h\|_\infty(1-|z|^2)^\alpha \frac{|f(z)|}{|z|}\\
=& (1-|z|^2)^\alpha \frac{|f(z)|}{|z|(1-|z|)^\beta}\sum_{j=1}^{k} |a_k| 
\\&+ \|h\|_\infty(1-|z|^2)^\alpha \frac{|f(z)|}{|z|},
 \end{align*}
 where $\|\cdot\|_\infty$ stands for the classical sup norm.
From here we can further proceed as in the proof of Theorems \ref{1theorem2.2}, \ref{1theorem2.3} and \ref{1theorem2.4} according
to the condition on $\alpha$ and $\beta$.
\end{proof}
 From the statement of Theorems \ref{1theorem2.2}, \ref{1theorem2.3} and \ref{1theorem2.4}
 we can conclude that the Alexander operator is a bounded linear operator 
 from $\mathcal{B}_\alpha^0$ to $\mathcal{B}_\alpha^0$. But in addition, we discuss the exact operator
 norm for the Alexander operator from $\mathcal{B}_\alpha^0$ to $\mathcal{B}_\alpha^0$ in the following Theorem.
\begin{theorem}
 The Alexander operator is a bounded linear operator from $\mathcal{B}_\alpha^0$ to $\mathcal{B}_\alpha^0$ with operator norm $1$. 
\end{theorem}
\begin{proof}
Suppose that $f \in \mathcal{B}_\alpha^0$. Then we estimate
\begin{align*}
 |f(z)|      &= \bigg|z\int_0^1 f'(zt)dt\bigg| \leq |z|\int_0^1 |f'(zt)|dt.
\end{align*}
Multipling by $(1-|z|^2)^{\alpha}$ on both sides, we have
\begin{align*}
(1-|z|^2)^{\alpha} |f(z)| &\leq |z|(1-|z|^2)^{\alpha}\int_0^1 \frac{(1-|z|^2|t|^2)^{\alpha}}{(1-|z|^2|t|^2)^{\alpha}}|f'(zt)|dt.  
\end{align*}
By the definition of $\alpha$-Bloch space, we estimate
\begin{align*}
 (1-|z|^2)^{\alpha} |f(z)| &\leq |z|\|f\|_{\mathcal{B}_\alpha} \int_0^1 \frac{(1-|z|^2)^{\alpha}}{(1-|z|^2|t|^2)^{\alpha}}dt.
\end{align*}
Since $ (1-|z|^2)\leq (1-|z|^2|t|^2)$, we obtain
\begin{align*}             
(1-|z|^2)^{\alpha} \bigg|\frac{d}{dz}\int_0^z\frac{f(t)}{t}dt\bigg| &\leq \|f\|_{\mathcal{B}_\alpha}.
\end{align*}
Here $z$ is arbitrary point in $\mathbb{D}$, therefore
$$
 \|C_0(f)\|_{\mathcal{B}_\alpha}\leq \|f\|_{\mathcal{B}_\alpha}.
$$

If we choose the identity function $f(z)=z$, then we obtain the exact operator norm $1$, equivalently, we say that $\|C_0\|_{\mathcal{B}_\alpha} = 1$.
\end{proof}

\subsection*{Counterexamples}
We just proved that for either of the cases $\beta \leq \alpha<1$, $\beta \leq 1<\alpha$ and $\beta<\alpha=1$,
the $\beta$-Ces\`{a}ro operator is a bounded linear operator from $\mathcal{B}_{\alpha}^0$ to $\mathcal{B}_{\alpha}^0$.
We now show that for the remaining cases: $\beta>\alpha$, $\beta=\alpha\geq1$ and $1<\beta<\alpha$,
the $\beta$-Ces\`{a}ro operators need not be bounded as the following counterexamples show.
\begin{example}\label{1ex2.6} Let $f(z)=z$, then $f\in \mathcal{B}_\alpha^0 $, for $\beta > \alpha$ and we have
\begin{align*}
(1-|z|^2)^{\alpha} \bigg|\frac{z}{z(1-z)^\beta}\bigg|&=(1+|z|)^{\alpha} \frac{(1-|z|)^{\alpha}}{|(1-z)|^\beta}.
\end{align*}
For $z=t\in (0,1)$, we obtain
\begin{align*}
(1-|z|^2)^{\alpha} \bigg|\frac{z}{z(1-z)^\beta}\bigg|&=(1+t)^{\alpha} \frac{(1-t)^{\alpha}}{(1-t)^\beta}= \frac{(1+t)^{\alpha}}{(1-t)^{\beta-\alpha}}.
\end{align*}
As $t$ tends to $1$, right hand side term tends to $\infty$. Therefore, the $\beta$-Ces\`{a}ro operator is an unbounded linear operator from $\mathcal{B}_{\alpha}^0$ to $\mathcal{B}_{\alpha}^0 $, for $\beta > \alpha$.
\end{example}
\begin{example}\label{1ex2.7} Let $f(z)=\text{Log}(1-z)$, where a principal value of the branch of lograthm is chosen. Then $f\in \mathcal{B}_{\alpha}^0 $, for $\alpha\geq1$ and for $z=t \in (0,1)$, we have
\begin{align*}
(1-|z|^2)^{\alpha} \bigg|\frac{\log(1-z)}{z(1-z)^{\beta}}\bigg|&=(1-t^2)^{\alpha} \bigg|\frac{\log(1-t)}{t(1-t)^{\beta}}\bigg|
                        = \frac{(1+t)^{\alpha}}{(1-t)^{\beta-\alpha}}\bigg|\frac{\log(1-t)}{t}\bigg|.
\end{align*}
Then for $\beta\geq \alpha$, as $t$ tends to $1$, right hand side term diverges to $\infty$.
Therefore, the $\beta$-Ces\`{a}ro operator is an unbounded linear operator from $\mathcal{B}_{\alpha}^0$ to $\mathcal{B}_{\alpha}^0$, for $\beta \geq \alpha\geq1$.
\end{example}
\begin{remark}
  By using the conclusion of Examples $\ref{1ex2.6}$ and $\ref{1ex2.7}$, we are able to conclude that, 
the Ces\`{a}ro operator is an unbounded linear operator from $\mathcal{B}_{\alpha}^0 $ to $\mathcal{B}_{\alpha}^0 $, for $\alpha \leq1$.
\end{remark}

\begin{example}\label{1ex2.8} Let $f(z)=z/(1-z)^\alpha$, for $\alpha > 0$, then $f\in \mathcal{B}_{\alpha+1}^0 $ and we have
\begin{align*}
(1-|z|^2)^{\alpha+1} \bigg|\frac{z}{z(1-z)^{\alpha+\beta}}\bigg|&=(1-|z|^2)^{\alpha+1} \frac{1}{|(1-z)|^{\alpha+\beta}}.
\end{align*}
For $z=t \in (0,1)$, then it yields
\begin{align*}
(1-|z|^2)^{\alpha+1} \bigg|\frac{z}{z(1-z)^{\alpha+\beta}}\bigg|&=(1+t)^{\alpha+1} \frac{(1-t)^{\alpha+1}}{(1-t)^{\alpha+\beta}}= \frac{(1+t)^{\alpha+1}}{(1-t)^{\beta-1}}.
\end{align*}
Then for $\beta>1$, as $t$ tends to $1$, right hand side term approaches to $\infty$. 
Therefore, the $\beta$-Ces\`{a}ro operator is unbounded linear operator from $\mathcal{B}_{\alpha+1}^0$ to $\mathcal{B}_{\alpha+1}^0$, for $\beta > 1$.
\end{example}

\section{\textbf{Compactness of the $\beta$-Ces\`{a}ro operators}}\label{1sec3}
In this section, we discuss compactness of the $\beta$-Ces\`{a}ro operators, for $\beta<\alpha<1$, $\beta<1<\alpha$ and $\beta<\alpha=1$, and for its generalization with the help of Lemma \ref{1lemma3.2}. 
However, the same problem for the cases $\beta=\alpha<1$ and $\beta=1<\alpha$ will be investigated in the next section.
Before going to the equivalent condition for compactness of the generalized $\beta$-Ces\`{a}ro operators (Lemma \ref{1lemma3.2}),
we have the following Lemma, which is used to prove the subsequent lemma.   

\begin{lemma}\label{1lemma3.1}
The generalized $\beta$-Ces\`{a}ro operators mapping from $\mathcal{B}_\alpha^0$ to $\mathcal{B}_\alpha^0$ are continuous linear operators in the topology of uniform 
convergence on every compact subset of $\mathbb{D}$.	
\end{lemma}
\begin{proof}
Let $f_m$ be a sequence in $\mathcal{B}_\alpha^0$ which converges to $f$ uniformly on every compact subset of $\mathbb{D}$.
By the Weierstrass theorem for sequences, $f'_m$ converges to $f'$ uniformly on every compact subset of $\mathbb{D}$.
On the other hand, we have
\begin{equation}\label{1eq3.1}
\bigg|\frac{f_m(z)-f(z)}{z}\bigg| \leq \int_0^1 |f'_m(zt)-f'(zt)|dt.
\end{equation}
 Consequently, $f_m(z)/z$ converges to $f(z)/z$ uniformly on every compact subset of $\mathbb{D}$,
 which implies that $C_{g_\beta} (f_m)$ converges to $C_{g_\beta} (f)$ uniformly on every compact subset of $\mathbb{D}$.  
\end{proof}
   
\begin{lemma}\label{1lemma3.2}
	The generalized $\beta$-Ces\`{a}ro operators mapping from $\mathcal{B}_\alpha^0$ to $\mathcal{B}_\alpha^0$ are compact if and only if for every bounded sequence 
	$(f_m)$ in $\mathcal{B}_\alpha^0$ which converges to $0$ uniformly on every compact subset of $\mathbb{D}$ we have $\lim\limits_{m\rightarrow \infty}\|C_{g_\beta} f_m\|_{\mathcal{B}_\alpha}=0$.	
\end{lemma}
\begin{proof}
 The proof of this lemma follows from Lemma \ref{1lemma3.1} and \cite[Lemma 3]{Stevic06}.
\end{proof}
With the help of Lemma \ref{1lemma3.2}, we prove the following theorems.
\begin{theorem}\label{1theorem3.3}
  The generalized Alexander operator is a compact linear operator from $\mathcal{B}_{\alpha}^0$ to $\mathcal{B}_{\alpha}^0$, for $\alpha>0$.
\end{theorem}
\begin{proof}
Suppose $(f_m)$ is a sequence in $\mathcal{B}_\alpha^0$, which converges to 0 uniformly on every compact subset of $\mathbb{D}$ and also bounded such that there exists a constant $M\in\mathbb{N}$ with $\|f_m\|_{\mathcal{B}_\alpha}\leq M$. 
We need to show that $\lim\limits_{m\to \infty}\|C_{g_0} f_m\|_{\mathcal{B}_\alpha}=0$.

Let $(s_k)_{k\in \mathbb{N}}$ be a sequence which increasingly converges to $1$. We have
\begin{align}
\lim\limits_{m\rightarrow \infty}\|C_{g_0} f_m\|_{\mathcal{B}_\alpha}
           =&\lim\limits_{m\rightarrow \infty}\sup\limits_{|z|<1} (1-|z|^2)^\alpha \bigg|\frac{f_m(z)g_0(z)}{z}\bigg| \nonumber\\
           \leq & \lim\limits_{m\rightarrow \infty}\sup\limits_{|z|\leq s_k} (1-|z|^2)^\alpha \bigg|\frac{f_m(z)g_0(z)}{z}\bigg| \\
           &+\lim\limits_{m\rightarrow \infty}\sup\limits_{s_k<|z|<1} (1-|z|^2)^\alpha \bigg|\frac{f_m(z)g_0(z)}{z}\bigg| \nonumber.
\end{align}
From \eqref{1eq3.1}, $f_m(z)g_0(z)/z$ converges to 0 uniformly on every compact subset of $\mathbb{D}$. Thus, we obtain
\begin{equation}\label{1eq3.2}
\lim\limits_{m\rightarrow \infty}\sup\limits_{|z|\leq s_k} (1-|z|^2)^\alpha \bigg|\frac{f_m(z)g_0(z)}{z}\bigg| = 0,
\end{equation}
for each $k\in \mathbb{N}$.

We consider three cases here.
\subsection*{Case (i)} Assume that $\alpha<1$. Then from \eqref{1eq2.3}, we have
 \begin{equation}
  |f_m(z)|\leq \|f_m\|_{\mathcal{B}_\alpha}\frac{1}{1-\alpha}\nonumber.
 \end{equation}
This yields
 $$
 \sup\limits_{s_k<|z|<1}(1-|z|^2)^\alpha \bigg|\frac{f_m(z)g_0(z)}{z}\bigg|
                         \leq \sup\limits_{s_k<|z|<1}\frac{(1-|z|^2)^\alpha}{|z|(1-\alpha)} |g_0(z)|\|f_m\|_{\mathcal{B}_\alpha}.
                         $$
                         Equivalently, we obtain                       
                         \begin{equation}
\label{1eq3.3}                       
\sup\limits_{s_k<|z|<1}(1-|z|^2)^\alpha \bigg|\frac{f_m(z)g_0(z)}{z}\bigg|   \leq M\frac{(1-s_k^2)^\alpha}{s_k(1-\alpha)} \|g_0\|_\infty,
 \end{equation}
which tends to $0$ as $k \rightarrow \infty$.
 
\subsection*{Case (ii)} Suppose that $\alpha=1$. It follows from \eqref{1eq2.4} that
$$
|f_m(z)| \leq \frac{\|f_m\|_{\mathcal{B}_1}}{2}\log \bigg(\frac{1+|z|}{1-|z|}\bigg)\nonumber.
$$
Then we have
\begin{align}\label{1eq3.4}
&\nonumber\sup\limits_{s_k<|z|<1}(1-|z|^2) \bigg|\frac{f_m(z)g_0(z)}{z}\bigg|\\
&\hspace{2cm}\leq M \sup\limits_{s_k<|z|<1}\frac{(1-|z|^2)}{2|z|}\log \bigg(\frac{1+|z|}{1-|z|}\bigg) |g_0(z)|,
\end{align}
which tends to 0 as $k \rightarrow \infty$.

\subsection*{Case (iii)} Assume that $\alpha>1$. Then from \eqref{1eq2.5}, we have
$$
 |f_m(z)|\leq \frac{\|f_m\|_{\mathcal{B}_\alpha}}{\alpha-1} \bigg(\frac{1}{(1-|z|)^{\alpha-1}}-1\bigg).
$$
This leads to
\begin{align}\label{1eq3.5}
&\nonumber\sup\limits_{s_k<|z|<1}(1-|z|^2)^\alpha \bigg|\frac{f_m(z)g_0(z)}{z}\bigg|\\
&\hspace{1cm}\leq M \sup\limits_{s_k<|z|<1}\frac{(1-|z|^2)^\alpha}{|z|(\alpha-1)}\bigg(\frac{1}{(1-|z|)^{\alpha-1}}-1\bigg) |g_0(z)|.
\end{align}
The right hand side quantity tends to 0 as $k \rightarrow \infty$.
Thus, from \eqref{1eq3.2}, \eqref{1eq3.3}, \eqref{1eq3.4} and \eqref{1eq3.5}, we conclude that
$$
 \lim\limits_{m\rightarrow \infty} \|C_{g_0} f_m\|_{\mathcal{B}_\alpha} = 0.
 $$
 The proof of our theorem is complete.
\end{proof}
\begin{theorem}\label{1theorem3.4}
The $\beta$-Ces\`{a}ro operator is a compact linear operator from $\mathcal{B}_{\alpha}^0$ to $\mathcal{B}_{\alpha}^0$, either for $\beta<\alpha<1$ or $\beta<1<\alpha$ or $\beta<\alpha=1$.
\end{theorem}
\begin{proof}
 Suppose $(f_m)$ is a bounded sequence in $\mathcal{B}_\alpha^0$ and also converges to $0$ uniformly on compact subsets of $\mathbb{D}$. 
Let $(s_k)_{k\in \mathbb{N}}$ be a sequence which increasingly converges to $1$. We compute
\begin{align}
\lim\limits_{m\rightarrow \infty}\|C_\beta f_m\|_{\mathcal{B}_\alpha}
           =&\lim\limits_{m\rightarrow \infty}\sup\limits_{|z|<1} (1-|z|^2)^\alpha \bigg|\frac{f_m(z)}{z(1-z)^\beta}\bigg| \nonumber\\
           \leq& \lim\limits_{m\rightarrow \infty}\sup\limits_{|z|\leq s_k} \frac{(1-|z|^2)^\alpha}{(1-|z|)^\beta} \bigg|\frac{f_m(z)}{z}\bigg|\\ 
           &+\lim\limits_{m\rightarrow \infty}\sup\limits_{s_k<|z|<1} \frac{(1-|z|^2)^\alpha}{(1-|z|)^\beta}\bigg|\frac{f_m(z)}{z}\bigg| \nonumber.
\end{align}
To compute the second term in right hand side, we need to consider three cases on $\alpha$.
\subsection*{Case (i)} Consider $\beta<\alpha<1$. Then we have
$$
\sup\limits_{s_k<|z|<1}(1-|z|^2)^\alpha \bigg|\frac{f_m(z)}{z(1-z)^\beta}\bigg| \leq  \sup\limits_{s_k<|z|<1}\frac{(1-|z|^2)^\alpha}{|z|(1-|z|)^\beta (1-\alpha)}\|f_m\|_{\mathcal{B}_\alpha},
$$
since \eqref{1eq2.3} gives $|f_m(z)|\leq (1-\alpha)^{-1}\|f_m\|_{\mathcal{B}_\alpha}$. Further, we obtain
\begin{align}\label{1eq3.8}
 \sup\limits_{s_k<|z|<1}(1-|z|^2)^\alpha \bigg|\frac{f_m(z)}{z(1-z)^\beta}\bigg|&\leq 2^{\alpha}\frac{(1-s_k)^{\alpha-\beta}}{s_k(1-\alpha)}\|f_m\|_{\mathcal{B}_\alpha}.
\end{align}
The above right hand side term tends to $0$ as $k \rightarrow \infty$.
\subsection*{Case (ii)} Assume that $\beta<\alpha=1$. It follows that
\begin{align}\label{1eq3.9}
&\sup\limits_{s_k<|z|<1}(1-|z|^2) \bigg|\frac{f_m(z)}{z(1-z)^\beta}\bigg|\\
&\hspace{2cm}\leq \frac{\|f_m\|_{\mathcal{B}_1}}{s_k} \sup\limits_{s_k<|z|<1}(1-|z|)^{1-\beta}\log \bigg(\frac{1+|z|}{1-|z|}\bigg),\nonumber
\end{align}
where we used the inequality 
$$
|f_m(z)|\leq \frac{\|f_m\|_{\mathcal{B}_1}}{2}\log \bigg(\frac{1+|z|}{1-|z|}\bigg),
$$
which is due to \eqref{1eq2.4}.
Now, as $k \rightarrow \infty$ 
\begin{align*}
 \sup\limits_{s_k<|z|<1}(1-|z|)^{1-\beta}\log \bigg(\frac{1+|z|}{1-|z|}\bigg)
\end{align*}
tends to 0.
\subsection*{Case (iii)} Suppose that $\alpha>1>\beta$. Then from \eqref{1eq2.5}, we obtain
\begin{align}
	|f_m(z)|&\leq \frac{\|f_m\|_{\mathcal{B}_\alpha}}{\alpha-1} \bigg(\frac{1}{(1-|z|)^{\alpha-1}}-1\bigg)\nonumber.
\end{align}
It follows that
\begin{align*}
&\sup\limits_{s_k<|z|<1}(1-|z|^2)^\alpha \bigg|\frac{f_m(z)}{z(1-z)^\beta}\bigg|\\
	&\hspace{1cm}\leq  \sup\limits_{s_K<|z|<1}\frac{(1-|z|^2)^\alpha}{|z|(\alpha-1)}\bigg(\frac{1}{(1-|z|)^{\beta+\alpha-1}}-\frac{1}{(1-|z|)^{\beta}}\bigg)\|f_m\|_{\mathcal{B}_\alpha}.
\end{align*}
This equivalently gives
\begin{align}
\label{1eq3.10}
&\sup\limits_{s_k<|z|<1}(1-|z|^2)^\alpha \bigg|\frac{f_m(z)}{z(1-z)^\beta}\bigg|\nonumber\\
&\hspace{.5cm}\leq \sup\limits_{s_k<|z|<1}\frac{(1+|z|)^\alpha}{|z|(\alpha-1)}\Big((1-|z|)^{1-\beta}-(1-|z|)^{\alpha-\beta}\Big)\|f_m\|_{\mathcal{B}_\alpha}.
\end{align}
The right hand side quantity of \eqref{1eq3.10} tends to $0$ as $k \rightarrow \infty$.
	
By \eqref{1eq3.1}, $f_m(z)/z$ converges to $0$ uniformly on compact subsets of $\mathbb{D}$. This leads to
\begin{equation}\label{1eq3.7}
\lim\limits_{m\rightarrow \infty}\sup\limits_{|z|\leq s_k} \frac{(1-|z|^2)^\alpha}{(1-|z|)^\beta} \bigg|\frac{f_m(z)}{z}\bigg| = 0,
\end{equation}
for each $k\in \mathbb{N}$.

	From \eqref{1eq3.7}, \eqref{1eq3.8}, \eqref{1eq3.9} and \eqref{1eq3.10}, we obtain
	$$
	\lim\limits_{m\rightarrow \infty} \|C_\beta f_m\|_{\mathcal{B}_\alpha} = 0
	$$
	for $\beta<\alpha<1$, $\beta<1<\alpha$ and $\beta<\alpha=1$. This is what we wanted to show.
\end{proof}
\begin{corollary}\label{1corollary3.1}
The generalized $\beta$-Ces\`{a}ro operator is a compact linear operator from $\mathcal{B}_{\alpha}^0$ to $\mathcal{B}_{\alpha}^0$, either for $\beta<\alpha<1$ or $\beta<1<\alpha$ or $\beta<\alpha=1$.
\end{corollary}
\begin{proof}
	The proof of Corollary \ref{1corollary3.1} follows steps given in the proof of Corollary \ref{1corollary2.2} and then we proceed like proof of Theorem \ref{1theorem3.3} and Theorem \ref{1theorem3.4}. 
\end{proof}

\section{\textbf{ Essential norm of the $\beta$-Ces\'{a}ro operators}}\label{1sec4}
This section is devoted to obtaining the essential norm of the $\beta$-Ces\`{a}ro operators.
First we recall the concept of essential norm.
 Let $X$ and $Y$ be Banach spaces and $T:X\rightarrow Y$ be a bounded linear operator. The {\em essential norm}
 of the operator $T:X\rightarrow Y$, denoted by $\|T\|_e$, is defined as follows
 \begin{equation}\label{Sec4Eq1}
  \|T\|_e = \inf\{\|T+K\|:\text{K is a compact operator from X to Y} \},
 \end{equation}
where $\|.\|$ denotes the operator norm. The following remark is a
direct consequence of \eqref{Sec4Eq1}.

\begin{remark}\label{1remark4.1}
It is well-known that the set of all compact operators from a normed linear space to a Banach space is a closed subset of the set of bounded operators. Using this fact together with \eqref{Sec4Eq1}, one can easily show that 
an operator T is compact if and only if $ \|T\|_e=0$.
\end{remark}

The compactness of the generalized $\beta$-Ces\`{a}ro operator is studied directly in the
previous section; however the situations when $\beta=\alpha<1$ and 
$\beta=1<\alpha$ could not be handled directly. In this section, the concept of
essential norm played a crucial role in handling these unsolved situations.

\begin{theorem}\label{1theorem4.2}
  The essential norm of the generalized Alexander operator from $\mathcal{B}_{\alpha}^0$ to $\mathcal{B}_{\alpha}^0$ is $0$.
\end{theorem}
\begin{proof}
Consider the operator defined on $\mathcal{B}_{\alpha}^0$ by
$$
 K_{g_0}^{s_k} (f)(z) = \int_0^z \frac{f(s_k t)g_0(t)}{t}dt
$$
where $(s_k)_{k \in \mathbb{N}}$ is a increasing sequence converging to $1$ and $g_0$ is a bounded analytic 
function in $\mathbb{D}$.

Suppose $(f_m)_{m\in \mathbb{N}}$ is a bounded sequence in $\mathcal{B}_\alpha^0$ which converges to 0 uniformly on every compact subset of $\mathbb{D}$.
Then we see that 
\begin{align*}
 \sup\limits_{|z|<1}(1-|z|^2)^\alpha \bigg|\frac{f_m(s_k z)g_0(z)}{z}\bigg|
                          &\leq\sup\limits_{|z|<1}(1-|s_k z|^2)^\alpha \bigg|\frac{f_m(s_k z)g_0(z)}{s_k z}\bigg|\\
                          &\leq \|g_0\|_\infty \sup\limits_{|z|\leq s_k}(1-|z|^2)^\alpha \bigg|\frac{f_m(z)}{z}\bigg| \rightarrow 0
\end{align*}
as $m \rightarrow \infty$. Hence by Lemma \ref{1lemma3.2}, $K_{g_0}^{s_k}$ is compact for each $k \in \mathbb{N}$.

Let $\lambda \in (0,1)$ be fixed for the moment. Then we have
 \begin{align*}
  \|C_{g_0} - K_{g_0}^{s_k}\|_{\mathcal{B}_\alpha} =&\sup\limits_{\|f\|_{\mathcal{B}_\alpha}\leq 1}\sup\limits_{|z|<1}(1-|z|^2)^\alpha \bigg|\frac{f(z)g_0(z)}{z}-\frac{f(s_k z)g_0(z)}{z}\bigg|\\
                          \leq& \sup\limits_{\|f\|_{\mathcal{B}_\alpha}\leq 1}\sup\limits_{|z|\leq\lambda}(1-|z|^2)^\alpha \bigg|\frac{f(z)g_0(z)}{z}-\frac{f(s_k z)g_0(z)}{z}\bigg|\\
                          &+ \sup\limits_{\|f\|_{\mathcal{B}_\alpha}\leq 1}\sup\limits_{\lambda<|z|<1}(1-|z|^2)^\alpha \bigg|\frac{f(z)g_0(z)}{z}-\frac{f(s_k z)g_0(z)}{z}\bigg|.
 \end{align*}
 By using the classical mean value theorem and definition of $\alpha$-Bloch space, we obtain
 \begin{align*}
& \sup\limits_{\|f\|_{\mathcal{B}_\alpha}\leq 1}\sup\limits_{|z|\leq\lambda}(1-|z|^2)^\alpha \bigg|\frac{f(z)g_0(z)}{z}-\frac{f(s_k z)g_0(z)}{z}\bigg|\\
 &\hspace{2cm}\leq \sup\limits_{\|f\|_{\mathcal{B}_\alpha}\leq 1}\sup\limits_{|z|\leq\lambda}(1-|z|^2)^\alpha (1-s_k)|g_0(z)|\sup\limits_{|w|\leq\lambda}|f'(w)|.
  \end{align*}
  It follows that
  \begin{align} \label{1eq4.5}
&  \sup\limits_{\|f\|_{\mathcal{B}_\alpha}\leq 1}\sup\limits_{|z|\leq\lambda}(1-|z|^2)^\alpha \bigg|\frac{f(z)g_0(z)}{z}-\frac{f(s_k z)g_0(z)}{z}\bigg|\\
&\hspace{4cm} \leq \frac{(1-s_k)}{(1-\lambda^2)^\alpha}\|g_0\|_\infty\sup\limits_{\|f\|_{\mathcal{B}_\alpha}\leq 1}\|f\|_{\mathcal{B}_\alpha}\nonumber,
 \end{align}
 which tends to $0$ as $k\rightarrow \infty$.
 
 We consider the following cases to complete our proof.
 \subsection*{Case (i)} Assume that $\alpha<1$. Then from \eqref{1eq2.3}, we have
 \begin{equation}
  |f(z)|\leq \|f\|_{\mathcal{B}_\alpha}\frac{1}{1-\alpha}\nonumber.
 \end{equation}
 It follows that
\begin{align*}
 & \sup\limits_{\|f\|_{\mathcal{B}_\alpha}\leq1}\sup\limits_{\lambda<|z|<1}(1-|z|^2)^\alpha \bigg|\frac{f(z)g_0(z)}{z}-\frac{f(s_k z)g_0(z)}{z}\bigg|\\
                         &\hspace{60mm}\leq 2 \sup\limits_{\lambda<|z|<1}\frac{(1-|z|^2)^\alpha}{|z|(1-\alpha)} |g_0(z)|.
\end{align*}
Thus, we obtain
\begin{align}
\label{1eq4.6}         
& \nonumber\sup\limits_{\|f\|_{\mathcal{B}_\alpha}\leq 1}\sup\limits_{\lambda<|z|<1}(1-|z|^2)^\alpha \bigg|
\frac{f(z)g_0(z)}{z}-\frac{f(s_k z)g_0(z)}{z}\bigg|\\
& \hspace*{3cm}\leq 2\frac{(1-\lambda^2)^\alpha}{\lambda(1-\alpha)} \|g_0\|_\infty.
 \end{align}
The right hand side of \eqref{1eq4.6} tends to 0 as $\lambda \rightarrow 1$.

\subsection*{Case (ii)} Consider $\alpha=1$. From \eqref{1eq2.4}, we obtain
\begin{align}
|f(z)|&\leq \frac{\|f\|_{\mathcal{B}_1}}{2}\log \bigg(\frac{1+|z|}{1-|z|}\bigg)\nonumber.
\end{align}
This simplifies to
\begin{align}\label{1eq4.7}
&\nonumber\sup\limits_{\|f\|_{\mathcal{B}_1}\leq 1}\sup\limits_{\lambda<|z|<1}(1-|z|^2) \bigg|\frac{f(z)g_0(z)}{z}-\frac{f(s_k z)g_0(z)}{z}\bigg|\\
&\hspace*{2cm}\leq  \sup\limits_{\lambda<|z|<1}\frac{(1-|z|^2)}{|z|}\log \bigg(\frac{1+|z|}{1-|z|}\bigg) |g_0(z)|,
\end{align}
which tends to $0$ as $\lambda \rightarrow 1$.
\subsection*{Case (iii)} Suppose that $\alpha>1$. Then \eqref{1eq2.5} obtains
\begin{align}
 |f(z)|&\leq \frac{\|f\|_{\mathcal{B}_\alpha}}{\alpha-1} \bigg(\frac{1}{(1-|z|)^{\alpha-1}}-1\bigg)\nonumber.
\end{align}
It leads to
\begin{align}\label{1eq4.8}
&\nonumber\sup\limits_{\|f\|_{\mathcal{B}_\alpha}\leq 1}\sup\limits_{\lambda<|z|<1}(1-|z|^2)^\alpha \bigg|\frac{f(z)g_0(z)}{z}-\frac{f(s_k z)g_0(z)}{z}\bigg|\\
&\hspace{2cm}\leq 2 \sup\limits_{\lambda<|z|<1}\frac{(1-|z|^2)^\alpha}{|z|(\alpha-1)}\bigg(\frac{1}{(1-|z|)^{\alpha-1}}-1\bigg) |g_0(z)|.
\end{align}
The right hand side quantity of \eqref{1eq4.8} tends to 0 as $\lambda \rightarrow 1$.
Thus, from \eqref{1eq4.5}, \eqref{1eq4.6}, \eqref{1eq4.7} and \eqref{1eq4.8}, we obtain
$$
 \lim\limits_{k\rightarrow \infty} \|C_{g_0} - K_{g_0}^{s_k}\|_{\mathcal{B}_\alpha} = 0,
$$
 and the conclusion follows from the definition of essential norm.
\end{proof}

\begin{theorem}\label{1theorem4.3}
	The essential norm of the $\beta$-Ces\`{a}ro operator from $\mathcal{B}_{\alpha}^0$ to $\mathcal{B}_{\alpha}^0$ is $0$, either for $\beta<\alpha<1$ or $\beta<1<\alpha$ or $\beta<\alpha=1$.
\end{theorem}
\begin{proof}
	Consider the operator defined on $\mathcal{B}_{\alpha}^0$ by
	$$
	K_{s_k} (f)(z) = \int_0^z \frac{f(s_k t)}{t(1-t)^\beta}dt,
	$$
	where $(s_k)_{k \in \mathbb{N}}$ is an increasing sequence, which converges to $1$.
	
	Suppose $(f_m)_{m\in \mathbb{N}}$ is a bounded sequence in $\mathcal{B}_\alpha^0$,  which converges to $0$ uniformly on compact subsets of $\mathbb{D}$ and $\beta\leq\alpha$.
	Then we obtain
	\begin{align*}
	\sup\limits_{|z|<1}(1-|z|^2)^\alpha \bigg|\frac{f_m(s_k z)}{z(1-z)^\beta}\bigg|
	&\leq 2^\alpha\sup\limits_{|z|<1}(1-|z|)^{\alpha-\beta} \bigg|\frac{f_m(s_k z)}{s_k z}\bigg|\\
	&\leq 2^\alpha \sup\limits_{|z|\leq s_k}(1-|z|)^{\alpha-\beta} \bigg|\frac{f_m(z)}{z}\bigg|,
	\end{align*}
	which tends to 0 as $m \rightarrow \infty$. Hence by Lemma \ref{1lemma3.2}, $K_{s_k}$ is a compact operator for each $k \in \mathbb{N}$.
	
	Let $\lambda \in (0,1)$ be fixed for the moment. We compute
	\begin{align*}
	\|C_{\beta} - K_{s_k}\|_{\mathcal{B}_\alpha} =& \sup\limits_{\|f\|_{\mathcal{B}_\alpha}\leq 1}\sup\limits_{|z|<1}(1-|z|^2)^\alpha \bigg|\frac{f(z)}{z(1-z)^\beta}-\frac{f(s_k z)}{z(1-z)^\beta}\bigg|\\
	\leq& \sup\limits_{\|f\|_{\mathcal{B}_\alpha}\leq 1}\sup\limits_{|z|\leq\lambda}(1-|z|^2)^\alpha \bigg|\frac{f(z)}{z(1-z)^\beta}-\frac{f(s_k z)}{z(1-z)^\beta}\bigg|\\
	&+ \sup\limits_{\|f\|_{\mathcal{B}_\alpha}\leq 1}\sup\limits_{\lambda<|z|<1}(1-|z|^2)^\alpha \bigg|\frac{f(z)}{z(1-z)^\beta}-\frac{f(s_k z)}{z(1-z)^\beta}\bigg|.
	\end{align*}
	By the mean value theorem and definition of $\alpha$-Bloch space, it follows that
	\begin{align*}
	&\sup\limits_{\|f\|_{\mathcal{B}_\alpha}\leq 1}\sup\limits_{|z|\leq\lambda}(1-|z|^2)^\alpha \bigg|\frac{f(z)}{z(1-z)^\beta}-\frac{f(s_k z)}{z(1-z)^\beta}\bigg|\\
	&\hspace{3cm}\leq \sup\limits_{\|f\|_{\mathcal{B}_\alpha}\leq 1}\sup\limits_{|z|\leq\lambda}\frac{(1-|z|^2)^\alpha}{(1-|z|)^\beta} (1-s_k)\sup\limits_{|w|\leq\lambda}|f'(w)|.
	\end{align*}
	This obtains
	\begin{align}\label{1eq4.13}
	&\nonumber\sup\limits_{\|f\|_{\mathcal{B}_\alpha}\leq 1}\sup\limits_{|z|\leq\lambda}(1-|z|^2)^\alpha \bigg|\frac{f(z)}{z(1-z)^\beta}-\frac{f(s_k z)}{z(1-z)^\beta}\bigg|\\
	&\hspace{2cm}\leq \frac{(1-s_k)}{(1-\lambda^2)^\alpha (1-\lambda)^\beta}\sup\limits_{\|f\|_{\mathcal{B}_\alpha}\leq 1}\|f\|_{\mathcal{B}_\alpha},
	\end{align}
	which approaches to $0$ as $k\rightarrow\infty$.
	
	To pursue our goal, we will go through the following cases on $\alpha$ and $\beta$.
	
\subsection*{Case (i)} Consider $\beta<\alpha<1$. Then from \eqref{1eq2.3}, $|f(z)|\leq (1-\alpha)^{-1}\|f\|_{\mathcal{B}_\alpha}$, it follows that
	\begin{align}\label{1eq4.14}  
	&\nonumber\sup\limits_{\|f\|_{\mathcal{B}_\alpha}\leq 1}\sup\limits_{\lambda<|z|<1}(1-|z|^2)^\alpha \bigg|\frac{f(z)}{z(1-z)^\beta}-\frac{f(s_k z)}{z(1-z)^\beta}\bigg|\\
	&\hspace{1cm}\leq 2 \sup\limits_{\lambda<|z|<1}\frac{(1-|z|^2)^\alpha}{|z|(1-|z|)^\beta (1-\alpha)} \nonumber \\
	                        &\hspace{1cm}\leq 2^{\alpha+1}\frac{(1-\lambda)^{\alpha-\beta}}{\lambda(1-\alpha)},
	\end{align}
	which tends to $0$ as $\lambda \rightarrow 1$.
\subsection*{Case (ii)} Assume that $\beta<\alpha=1$.
First we see that
\begin{align*}
&\sup\limits_{\|f\|_{\mathcal{B}_1}\leq 1}\sup\limits_{\lambda<|z|<1}(1-|z|^2) \bigg|\frac{f(z)}{z(1-z)^\beta}-\frac{f(s_k z)}{z(1-z)^\beta}\bigg|\\
&\hspace{40mm}\leq\sup\limits_{\lambda<|z|<1}\frac{(1-|z|^2)}{|z|(1-|z|)^\beta}\log \bigg(\frac{1+|z|}{1-|z|}\bigg).
\end{align*}
 The above inequality easily follows from \eqref{1eq2.4}. This simplifies to
\begin{align}\label{1eq4.15}
&\nonumber\sup\limits_{\|f\|_{\mathcal{B}_1}\leq 1}\sup\limits_{\lambda<|z|<1}(1-|z|^2) \bigg|\frac{f(z)}{z(1-z)^\beta}-\frac{f(s_k z)}{z(1-z)^\beta}\bigg|\\
&\hspace{2cm}\leq \frac{2}{\lambda} \sup\limits_{\lambda<|z|<1}(1-|z|)^{1-\beta}\log \bigg(\frac{1+|z|}{1-|z|}\bigg).
\end{align}
The right hand side quantity of \eqref{1eq4.15} tends to 0 as $\lambda \rightarrow 1$.

\subsection*{Case (iii)} Suppose that $\alpha>1>\beta$. From \eqref{1eq2.5}, we have
	\begin{align}
	|f(z)|&\leq \frac{\|f\|_{\mathcal{B}_\alpha}}{\alpha-1} \bigg(\frac{1}{(1-|z|)^{\alpha-1}}-1\bigg)\nonumber.
	\end{align}
	Then we obtain
	\begin{align}\label{1eq4.16}
	&\nonumber\sup\limits_{\|f\|_{\mathcal{B}_\alpha}\leq 1}\sup\limits_{\lambda<|z|<1}(1-|z|^2)^\alpha \bigg|\frac{f(z)}{z(1-z)^\beta}-\frac{f(s_k z)}{z(1-z)^\beta}\bigg|\\
	&\hspace{1cm}\leq 2 \sup\limits_{\lambda<|z|<1}\frac{(1+|z|)^\alpha}{|z|(\alpha-1)}\bigg((1-|z|)^{1-\beta}-(1-|z|)^{\alpha-\beta}\bigg),
	\end{align}
	which tends to 0 as $\lambda \rightarrow 1$.
	
	From \eqref{1eq4.13}, \eqref{1eq4.14}, \eqref{1eq4.15} and \eqref{1eq4.16}, we thus obtain
	$$
	\lim\limits_{k\rightarrow \infty} \|C_\beta - K_{s_k}\|_{\mathcal{B}_\alpha} = 0
	$$
	for $\beta<\alpha<1$, $\beta<1<\alpha$ and $\beta<\alpha=1$. This completes the proof.
\end{proof}
\begin{corollary}\label{1corollary4.1}
	The essential norm of the generalized $\beta$-Ces\`{a}ro operator from $\mathcal{B}_{\alpha}^0$ to $\mathcal{B}_{\alpha}^0$ is 0, either for $\beta<\alpha<1$ or $\beta<1<\alpha$ or $\beta<\alpha=1$.
\end{corollary}
\begin{proof}
 As in the proof of Theorem \ref{1theorem4.3}, we similarly define
 $$
	K_{s_k} (f)(z) = \int_0^z \frac{f(s_k t)g_\beta(t)}{t}dt
$$
where $(s_k)_{k \in \mathbb{N}}$ is a increasing sequence, which converges to $1$.

To complete the proof, we follow steps given in the proof of Corollary \ref{1corollary2.2} and follow the proofs of Theorem \ref{1theorem4.2} and Theorem \ref{1theorem4.3}. 
\end{proof}
\begin{theorem}\label{1theorem4.4}
The essential norm of the $\beta$-Ces\`{a}ro operator from $\mathcal{B}_{\alpha}^0$ to $\mathcal{B}_{\alpha}^0$ is $0$, for $\beta=\alpha<1$.
\end{theorem}
\begin{proof}
 Let $\beta=\alpha-\frac{1}{n}$, then by Theorem \ref{1theorem4.3}, the $\beta$-Ces\`{a}ro operator from $\mathcal{B}_{\alpha}^0$ to $\mathcal{B}_{\alpha}^0$, for $\alpha<1$,
 is a compact linear operator for each $n \in \mathbb{N}$. Note that
 \begin{align}
 \|C_{\alpha} - C_{\alpha-\frac{1}{n}}\|_{\mathcal{B}_\alpha} &= \sup\limits_{\|f\|_{\mathcal{B}_\alpha}\leq 1}\sup\limits_{|z|<1}(1-|z|^2)^\alpha \bigg|\frac{f(z)}{z(1-z)^\alpha}-\frac{f(z)}{z(1-z)^{\alpha-\frac{1}{n}}}\bigg|.\nonumber
\end{align}
 For $f\in\mathcal{B}_\alpha^0$, $\alpha<1$, it is noted in Lemma~\ref{1lemma2.1} that $f$ is a bounded analytic function. Then we see that
 $$
 \|C_{\alpha} - C_{\alpha-\frac{1}{n}}\|_{\mathcal{B}_\alpha}  \leq \bigg\|\frac{f(z)}{z}\bigg\|_{\infty} \sup\limits_{|z|<1}\frac{(1-|z|^2)^\alpha}{(1-|z|)^\alpha} |1-(1-z)^{\frac{1}{n}}|,
 $$
 which is equivalent to
\begin{equation}\label{1eq4.18}  
\|C_{\alpha} - C_{\alpha-\frac{1}{n}}\|_{\mathcal{B}_\alpha}  
                                             \leq 2^\alpha \bigg\|\frac{f(z)}{z}\bigg\|_{\infty} \sup\limits_{|z|<1} |1-(1-z)^{\frac{1}{n}}|.
\end{equation}
For $z\in\mathbb{D}$ and $b$ on the unit circle, we compute
\begin{align*}
 |1-(1-bz)^{\frac{1}{n}}| & = \bigg|1-\exp\bigg(\frac{\text{Log}(1-bz)}{n}\bigg)\bigg| \\
 &= \bigg|1-\exp\bigg(\frac{\ln|1-bz|}{n}+\frac{i\text{Arg}(1-bz)}{n}\bigg)\bigg|\\
                         & = \bigg|1-\exp\bigg(\frac{\ln|1-bz|}{n}\bigg)+\exp\bigg(\frac{\ln|1-bz|}{n}\bigg)\\
                         & \hspace*{2cm}-\exp\bigg(\frac{\ln|1-bz|}{n}+\frac{i\text{Arg}(1-bz)}{n}\bigg)\bigg|.
\end{align*}
By triangle inequality, we have
\begin{align}
\nonumber|1-(1-bz)^{\frac{1}{n}}|  \leq & \bigg|1-\exp\bigg(\frac{\ln|1-bz|}{n}\bigg)\bigg|\\
&+\exp\bigg(\frac{\ln|1-bz|}{n}\bigg)\bigg|1-\exp\bigg(\frac{i\text{Arg}(1-bz)}{n}\bigg)\bigg|\nonumber\\
    \leq& \bigg|1-\exp\bigg(\frac{\ln 2}{n}\bigg)\bigg|\\
   &+\exp\bigg(\frac{\ln2}{n}\bigg)\bigg|1-\exp\bigg(\frac{i\text{Arg}(1-bz)}{n}\bigg)\bigg|.\nonumber
\end{align}
But
\begin{align}
 \bigg|1-\exp\bigg(\frac{i\text{Arg}(1-bz)}{n}\bigg)\bigg| =&\Big(\big(\cos(\text{Arg}(1-bz)/n)-1\big)^2\\
  &+ \sin^2 \big(\text{Arg}(1-bz)/n\big)\Big)^{1/2}.\nonumber
\end{align}
For $z\in \mathbb{D}$, it is clear that $-\pi/2\leq\text{Arg}(1-bz)\leq\pi/2$. It follows that
\begin{align}
 \bigg|1-\exp\bigg(\frac{i\text{Arg}(1-bz)}{n}\bigg)\bigg| &\leq \Big(\big(\cos(\pi/2n)-1\big)^2 + \sin^2 \big(\pi/2n\big)\Big)^{1/2},\nonumber
\end{align}
and hence
\begin{align*}
|1-(1-bz)^{\frac{1}{n}}|  \leq& \bigg|1-\exp\bigg(\frac{\ln 2}{n}\bigg)\bigg|\\&+\exp\bigg(\frac{\ln2}{n}\bigg)\Big(\big(\cos(\pi/2n)-1\big)^2 + \sin^2 \big(\pi/2n\big)\Big)^{1/2}.
\end{align*}
The right hand side quantity is independent of $z$ and tends to 0 as $n\rightarrow \infty$. Then we obtain 
\begin{equation} \label{1eq4.19}
 \lim \limits_{n\rightarrow \infty}\sup\limits_{|z|<1} |1-(1-bz)^{\frac{1}{n}}| = 0.
\end{equation}
From \eqref{1eq4.18} and \eqref{1eq4.19}, it follows that
$$
 \lim \limits_{n\rightarrow \infty}\|C_{\alpha} - C_{\alpha-\frac{1}{n}}\|_{\mathcal{B}_\alpha}  = 0,
$$
completing the proof.
\end{proof}
\begin{corollary}\label{1corollary4.2}
The essential norm of the generalized $\beta$-Ces\`{a}ro operator from $\mathcal{B}_{\alpha}^0$ to $\mathcal{B}_{\alpha}^0$ is $0$, for $\beta=\alpha<1$.
\end{corollary}
\begin{proof}
 To obtain the essential norm of the generalized $\beta$-Ces\`{a}ro operator, for $\beta=\alpha<1$, we compute
 \begin{align}
 \|C_{g_\alpha} - C_{g_{\alpha-\frac{1}{n}}}\|_{\mathcal{B}_\alpha} &= \sup\limits_{\|f\|_{\mathcal{B}_\alpha}\leq 1}\sup\limits_{|z|<1}(1-|z|^2)^\alpha \bigg|\frac{f(z)g_\alpha(z)}{z}-\frac{f(z)g_{\alpha-\frac{1}{n}}(z)}{z}\bigg|.\nonumber
\end{align}
To proceed further, we follow the steps given in the proofs of Corollary \ref{1corollary2.2} and Theorem \ref{1theorem4.4}.
\end{proof}

\begin{theorem}\label{1theorem4.5}
The essential norm of the $\beta$-Ces\`{a}ro operator from $\mathcal{B}_{\alpha}^0$ to $\mathcal{B}_{\alpha}^0$ is $0$, for $\beta=1<\alpha$.
\end{theorem}
\begin{proof}
 Let $\beta=1-\frac{1}{n}$, then by Theorem \ref{1theorem4.3}, the $\beta$-Ces\`{a}ro operator from $\mathcal{B}_{\alpha}^0$ to $\mathcal{B}_{\alpha}^0$, for $\alpha>1$,
 is a compact linear operator for each $n \in \mathbb{N}$. 
 
 Let $\lambda \in (0,1)$ be fixed for the moment. Then we have
 \begin{align}\label{1eq4.20}
 \nonumber\|C_1 - C_{1-\frac{1}{n}}\|_{\mathcal{B}_\alpha} 
 =& \sup\limits_{\|f\|_{\mathcal{B}_\alpha}\leq 1}\sup\limits_{|z|<1}(1-|z|^2)^\alpha \bigg|
 \frac{f(z)}{z(1-z)}-\frac{f(z)}{z(1-z)^{1-\frac{1}{n}}}\bigg|\\
 \leq& \sup\limits_{\|f\|_{\mathcal{B}_\alpha}\leq 1}\sup\limits_{|z|\leq\lambda}(1-|z|^2)^\alpha \bigg|
 \frac{f(z)}{z(1-z)}-\frac{f(z)}{z(1-z)^{1-\frac{1}{n}}}\bigg|\\
 &+ \sup\limits_{\|f\|_{\mathcal{B}_\alpha}\leq 1}\sup\limits_{|z|>\lambda}(1-|z|^2)^\alpha \bigg|
 \frac{f(z)}{z(1-z)}-\frac{f(z)}{z(1-z)^{1-\frac{1}{n}}}\bigg|.  \nonumber      
\end{align}
For $\alpha>1$, from \eqref{1eq2.5}, we have
	\begin{align}
	|f(z)|&\leq \frac{\|f\|_{\mathcal{B}_\alpha}}{\alpha-1} \bigg(\frac{1}{(1-|z|)^{\alpha-1}}-1\bigg)\nonumber.
	\end{align}
	We obtain 
\begin{align*}
& \sup\limits_{\|f\|_{\mathcal{B}_\alpha}\leq 1}\sup\limits_{|z|>\lambda}(1-|z|^2)^\alpha \bigg|\frac{f(z)}{z(1-z)}-\frac{f(z)}{z(1-z)^{1-\frac{1}{n}}}\bigg|\\
                    & \hspace*{1cm} \leq  \sup\limits_{\|f\|_{\mathcal{B}_\alpha}\leq 1} \sup\limits_{|z|>\lambda}\frac{(1-|z|^2)^\alpha}{|z|(1-|z|)} |1-(1-z)^{\frac{1}{n}}||f(z)|\\
                    &\hspace*{1cm}\leq \frac{2^\alpha}{\lambda} \sup\limits_{|z|>\lambda} |1-(1-z)^{\frac{1}{n}}|\Big(1-(1-|z|)^{\alpha-1}\Big)\\
                    &\hspace*{1cm}\leq \frac{2^\alpha}{\lambda} \sup\limits_{|z|>\lambda} |1-(1-z)^{\frac{1}{n}}|,
 \end{align*}
which  tends to $0$ as $n\rightarrow\infty$ due to \eqref{1eq4.19}. The first term of the right hand quantity of \eqref{1eq4.20} also tends to $0$ as $n\rightarrow\infty$. It yields	
$$
 \lim \limits_{n\rightarrow \infty}\|C_1 - C_{1-\frac{1}{n}}\|_{\mathcal{B}_\alpha}  = 0,
$$
which concludes the proof.
\end{proof}
\begin{corollary}
	The essential norm of the generalized $\beta$-Ces\`{a}ro operator from $\mathcal{B}_{\alpha}^0$ to $\mathcal{B}_{\alpha}^0$ is $0$, for $\beta=1<\alpha$.
\end{corollary}
\begin{proof}
 The proof follows the steps as in the proofs of Corollary \ref{1corollary4.2} and Theorem \ref{1theorem4.5}.
\end{proof}

\section{\textbf{Spectral Properties of the $\beta$-Ces\`{a}ro operators}}\label{1sec5}
In this section, we compute the (point) spectrum of the generalized $\beta$-Ces\`{a}ro operators 
mapping from $\mathcal{B}_{\alpha}^0 $ to $\mathcal{B}_{\alpha}^0 $, for $\alpha>0$.
For the concept of spectral analysis, we refer to \cite{Bela90,Kreyszig01,Limaye96}.

\begin{theorem}\label{1theorem5.1} For $\alpha\geq1$, the point spectrum of the generalized Ces\`{a}ro operator from $\mathcal{B}_{\alpha}^0 $ to $\mathcal{B}_{\alpha}^0 $ is
 \begin{align*}
 \sigma_P(C_{g_1}) &= \bigg \lbrace\frac{g_1(0)}{n},n \in \mathbb{N}: {\rm Re}\,\bigg(\frac{a_j}{g_1(0)}\bigg)\leq0,1\leq j\leq k\bigg \rbrace.
 \end{align*}
\end{theorem}
 \begin{proof} 
 We adopt the idea of the proof partially from the work of Albrecht, Miller, and Neumann \cite{Albrecht05}.
 Suppose that $f\in \mathcal{B}_{\alpha}^0\setminus\{0\}$. Write $f=z^n\psi$, where $n\geq1$ and $\psi$ is an analytic function in $\mathbb{D}$ with $\psi(0)\neq 0$. 
 We desire to show that $\psi\in\mathcal{B}_\alpha$.
 
 Suppose $(\lambda I-C_{g_\beta})f=0$. We obtain that
 \begin{align*}
  \lambda f(z)-\int_{0}^{z} \frac{f(w)}{w}g_\beta(w)dw &= 0,
 \end{align*}
 which is equivalent to
 \begin{align*}
   \lambda z^n\psi(z)-\int_{0}^{z} \frac{w^n\psi(w)}{w}g_\beta(w)dw &= 0.
\end{align*}
 Differentiating both sides with respect to nonzero $ z \in \mathbb{D}$, we obtain
 \begin{align*}
  (\lambda n z^{n-1}\psi(z)+ \lambda z^{n}\psi'(z)) - z^{n-1}\psi(z)g_\beta(z) &=0,
  \end{align*}
  which is equivalent to
 \begin{align}\label{1eq5.1}
  (\lambda n \psi(z)+ \lambda z\psi'(z)) - \psi(z)g_\beta(z) &=0.
 \end{align}
 By the continuity, \eqref{1eq5.1} also holds at $0$. Then at $z=0$ we have
 \begin{align*}
  \lambda n\psi(0) - \psi(0)g_\beta(0) &=0.
 \end{align*}
 On simplification, we obtain
 \begin{align*}
 \lambda &=\frac{g_\beta(0)}{n}.
 \end{align*}
 Then we have the possible point spectrum
 \begin{align*}
  \sigma_P(C_{g_\beta}) \subseteq \bigg \lbrace\frac{g_\beta(0)}{n}:n \in \mathbb{N}\bigg \rbrace.
 \end{align*}
If $g_\beta(0)=0$, then $\sigma_P(C_{g_\beta}) = \emptyset$.

We further need to obtain the condition under which the generalized Ces\`{a}ro operator have eigenvector corresponding to $g_1(0)/n$.
If $g_1(0)\neq 0$, then with $g_1(0)/n$, \eqref{1eq5.1} has a solution $\psi=c\psi_n$, where
\begin{align*}
 \psi_n(z)              &=\exp\bigg(\frac{n}{g_1(0)}\int_{0}^{z}\frac{g_1(w)-g_1(0)}{w}dw\bigg)\\
                        &=\exp\bigg(\frac{n}{g_1(0)}\int_{0}^{z}\frac{\sum_{j=1}^{k}\frac{a_j}{1-b_{j} w}+h(w)-\sum_{j=1}^{k}a_j-h(0)}{w}dw\bigg).
                        \end{align*}
                        This simplifies to
\begin{equation}\label{1eq5.2}          
\psi_n(z)              
=\exp\bigg(\frac{n}{g_1(0)}\sum_{j=1}^{k}\int_{0}^{z}\frac{a_{j} b_{j}}{1-b_{j} w}dw\bigg)\exp\bigg(\frac{n}{g_1(0)}\int_{0}^{z}\frac{h(w)-h(0)}{w}dw\bigg).
\end{equation}
Set
$$\phi(z) =\exp\bigg(\frac{n}{g_1(0)}\sum_{j=1}^{k}\int_{0}^{z}\frac{a_{j} b_{j}}{1-b_{j} w}dw\bigg)$$
and 
$$\eta(z)=\exp\bigg(\frac{n}{g_1(0)}\int_{0}^{z}\frac{h(w)-h(0)}{w}dw\bigg).
$$ 
Then we have by Schwarz Lemma
\begin{align}\label{1eq5.3}
\exp\bigg(-2\bigg|\frac{n}{g_1(0)}\bigg|\|h\|_\infty\bigg)\leq |\eta(z)|\leq \exp\bigg(2\bigg|\frac{n}{g_1(0)}\bigg|\|h\|_\infty\bigg). 
\end{align}
Next we show that $\phi$ is bounded. For this, we compute
\begin{align*}
 |\phi(z)|        =& \bigg|\exp\bigg(\frac{n}{g_1(0)}\sum_{j=1}^{k}\int_{0}^{z}\frac{a_{j} b_{j}}{1-b_{j} w}dw\bigg)\bigg|\\
                  =& \bigg|\exp\bigg(\sum_{j=1}^{k} -n\bigg(\cfrac{a_{j}}{g_1(0)}\bigg) \log(1-b_{j}z)\bigg)\bigg|.
                  \end{align*}
                  This is equivalent to
$$
 |\phi(z)|  = \left|\prod_{j=1}^{k}(1-b_{j}z)^{-n\bigg(\cfrac{a_{j}}{g_1(0)}\bigg)} \right|.
$$
To prove the boundedness of $ |\phi(z)| $, it is sufficient to show that for each $j$, $1\le j\le k$, 
the quantity $(1-b_{j}z)^{-n\big(\frac{a_{j}}{g_1(0)}\big)}$ is bounded in $\mathbb{D}$.
For this purpose, we see that
\begin{align*}
&(1-b_{j}z)^{-n\big(\frac{a_{j}}{g_1(0)}\big)}\\
   &\hspace{1cm}=\exp\bigg(-n\frac{a_{j}}{g_1(0)}\log(1-b_{j}z)\bigg) \\
   &\hspace{1cm}=\exp\bigg(-n\frac{a_{j}}{g_1(0)}\big(\ln|1-b_{j}z|+i \arg(1-b_{j}z)\big)\bigg)\\
   &\hspace{1cm}=\exp\bigg(-n\frac{a_{j}}{g_1(0)}\big(\ln|1-b_{j}z|\big)\bigg)\exp\bigg(-n\frac{a_{j}}{g_1(0)}i \arg(1-b_{j}z)\bigg).
\end{align*}
Now,
\begin{align*}
&\bigg|\exp\bigg(-n\frac{a_{j}}{g_1(0)}\big(\ln|1-b_{j}z|\big)\bigg)\bigg|\\
&\hspace*{1cm}=\bigg|\exp\bigg(\big(-n\ln|1-b_{j}z|\big)\bigg({\rm Re}\, \bigg(\frac{a_{j}}{g_1(0)}\bigg)+i\ {\rm Im}\,\bigg( \frac{a_{j}}{g_1(0)}\bigg)\bigg)\bigg)\bigg|\\
                                 &\hspace*{1cm}=\bigg|\exp\bigg(\big(-n\ln|1-b_{j}z|\big){\rm Re}\, \bigg(\frac{a_{j}}{g_1(0)}\bigg)\bigg)\bigg|.
\end{align*}
If ${\rm Re}\, \Big(\frac{a_{j}}{g_1(0)}\Big)\leq 0$, then $\Big|\exp\Big(\big(-n\ln|1-b_{j}z|\big){\rm Re}\, \Big(\frac{a_{j}}{g_1(0)}\Big)\Big)\Big|$ is bounded.
This implies that $\phi$ is bounded analytic function in $\mathbb{D}$.

Differentiating $\psi_n(z)$ with respect to $z$, we obtain
\begin{align*}
\psi_n^{'}(z)=&\frac{n}{g_1(0)}\Bigg(\frac{\sum_{j=1}^{k}\frac{a_j}{1-b_{j} z}+h(z)-\sum_{j=1}^{k}a_j-h(0)}{z}\Bigg)\psi_n(z)\nonumber\\
             =&\frac{n}{g_1(0)}\bigg(\sum_{j=1}^{k}\frac{a_{j} b_{j}}{1-b_{j} z}+\frac{h(z)-h(0)}{z}\bigg)\psi_n(z).
\end{align*}
Write 
$$
\rho(z)=\frac{n}{g_1(0)}\Big(\sum_{j=1}^{k}\frac{a_{j} b_{j}}{1-b_{j} z}+\frac{h(z)-h(0)}{z}\Big).
$$ 
Now we need to check under what condition 
$\psi_n(z) \in \mathcal{B}_{\alpha}$. For this purpose, we need to calculate $(1-|z|^2)^{\alpha}|\psi_n^{'}(z)|$.
Now
\begin{align*}
(1-|z|^2)^{\alpha}|\psi_n^{'}(z)|=&(1-|z|^2)^{\alpha}|\rho(z)\psi_n(z)|\\
                                 =&(1-|z|^2)^{\alpha}|\rho(z)\phi(z)\eta(z)|\\
                                 =&(1-|z|^2)^{\alpha}|\rho(z)||\phi(z)||\eta(z)|.
\end{align*}
For $\alpha \geq 1$,
$$
\sup\limits_{z \in \mathbb{D}}(1-|z|^2)^{\alpha}|\rho(z)|<\infty 
$$
and for $\alpha < 1$, this is unbounded as can be seen from Example \ref{1ex2.6}.
 As we have already seen that $\phi$ and $\eta$ are bounded analytic functions in $\mathbb{D}$, we obtain 
 $$
\sup\limits_{z \in \mathbb{D}}(1-|z|^2)^{\alpha}|\psi(z)|<\infty,
$$
as desired to have $f\in\mathcal{B}_\alpha^0$, $\alpha\ge 1$.
\end{proof}

If $g_1(w)=h(w)$ for each $w \in \mathbb{D}$ as in \eqref{1eq1.2}, 
then by \eqref{1eq5.2} and \eqref{1eq5.3} we establish
\begin{theorem} \label{1theorem5.2}
 The point spectrum of the generalized Alexander operator from $\mathcal{B}_{\alpha}^0$ to $\mathcal{B}_{\alpha}^0 $ is
\begin{align*}
\sigma_P(C_{g_0}) &= \bigg \lbrace\frac{g_0(0)}{n},n \in \mathbb{N}\bigg \rbrace.
\end{align*}
\end{theorem}

\begin{example} \label{1example5.3}
 Let $f_n(z)={z^n}/n$, for $n\in\mathbb{N}$. Clearly for each $n \in \mathbb{N}$, $f_n \in \mathcal{B}_{\alpha}^0$. 
 Define 
 \begin{align*}
 h_n(z)&=\frac{z^n}{n\|f_n\|_{\mathcal{B}_{\alpha}}}.
 \end{align*}
 We can easily obtain that $\|h_n\|_{\mathcal{B}_{\alpha}}=1$. We estimate
 \begin{align*}
  \|C_{g_0}(h_n)(z)\|_{\mathcal{B}_{\alpha}}&=\bigg\|\int_{0}^{z}\frac{h_n(t) g_0(t)}{t}dt\bigg\|_{\mathcal{B}_{\alpha}}\\
  &\leq\frac{\|g_0\|_\infty}{n\|f_n\|_{\mathcal{B}_{\alpha}}} \bigg\|\int_{0}^{z}t^{n-1}dt\bigg\|_{\mathcal{B}_{\alpha}}                                            
                                            =\frac{\|g_0\|_\infty}{n}  ,
\end{align*}
 which tends to $0$ as $n$ tends to $\infty$. This implies that $(h_n)$ is an approximate eigenvector with eigenvalue $0$. Definition of approximate eigenvector is found in
 \cite[chapter 12]{Bela90}. In other words, we also say that $0$ is the approximate eigenvalue of $C_{g_0}$.
\end{example}
 Using Theorem \ref{1theorem3.3} or Theorem \ref{1theorem4.2} together with Theorem \ref{1theorem5.2}, we can compute the spectrum of the generalized Alexander operator on 
 $\mathcal{B}_{\alpha}^0 $, which is stated in the following form.
 
 \begin{theorem}\label{1theorem5.3}
 The spectrum of the generalized Alexander operator from $\mathcal{B}_{\alpha}^0 $ to $\mathcal{B}_{\alpha}^0 $ is
\begin{align*}
\sigma(C_{g_0}) &= \bigg \lbrace\frac{g_0(0)}{n},n \in \mathbb{N}\bigg \rbrace \bigcup \bigg \lbrace 0 \bigg \rbrace,
\end{align*}
where $0$ is the approximate eigenvalue. 
\end{theorem}

\begin{remark}
The approximate eigenvalue stated in Theorem~\ref{1theorem5.3} is also discussed
in Example $\ref{1example5.3}$.
\end{remark}

The following theorem provides us the point spectrum of the $\beta$-Ces\`{a}ro operator from 
$\mathcal{B}_\alpha^0$ to itself for various choices of positive $\beta$. The
non-positive values of $\beta$ turns the operator into the generalized 
Alexander operator which is already covered in Theorem~\ref{1theorem5.3}.

\begin{theorem} 
 The point spectrum of the $\beta$-Ces\`{a}ro operator from $\mathcal{B}_{\alpha}^0 $ to $\mathcal{B}_{\alpha}^0 $,
 either for $0<\beta\leq\alpha<1$ or $0<\beta<1<\alpha$ or $0<\beta<\alpha=1$, is
\begin{align*}
\sigma_P(C_{g_\beta}) &= \bigg \lbrace\frac{g_\beta(0)}{n},n \in \mathbb{N}\bigg \rbrace.
\end{align*}
\end{theorem}
\begin{proof}
  If $g_\beta(0)\neq 0$, then with $g_\beta(0)/n$, \eqref{1eq5.1} has a solution $\psi=c\psi_n$, where
\begin{align*}
 \psi_n(z)              &=\exp\bigg(\frac{n}{g_\beta(0)}\int_{0}^{z}\frac{g_\beta(w)-g_\beta(0)}{w}dw\bigg)\\
                        &=\exp\bigg(\frac{n}{g_\beta(0)}\int_{0}^{z}\frac{\sum_{j=1}^{k}\frac{a_j}{(1-b_{j} w)^\beta}+h(w)-\sum_{j=1}^{k}a_j-h(0)}{w}dw\bigg).
                        \end{align*}
                        This is equivalent to
\begin{align}
\label{1eq5.10}         
&\psi_n(z)=\exp\bigg(\frac{n}{g_\beta(0)}\sum_{j=1}^{k}\int_{0}^{z}\frac{\frac{a_{j}}{(1-b_{j} w)^\beta}-a_j}{w}dw\bigg)\\
& \hspace{3cm}\exp\bigg(\frac{n}{g_\beta(0)}\int_{0}^{z}\frac{h(w)-h(0)}{w}dw\bigg).\nonumber
\end{align}
We have already proved in Theorem~\ref{1theorem5.1} that the second factor is bounded.
Hence, it remains to consider only the first factor here.
Recall that
$$
 \frac{1}{(1-b_{j} w)^\beta} = \sum_{n=0}^{\infty}\frac{\Gamma (n+\beta)}{\Gamma (n+1)\Gamma (\beta)}(b_j w)^{n},
$$
where $\Gamma$ is the classical Euler gamma function. We compute the integral 
 \begin{align}
  \int_{0}^{z}\frac{\frac{a_{j}}{(1-b_{j} w)^\beta}-a_j}{w}dw&= a_j\int_{0}^{z} \sum_{n=1}^{\infty}\frac{\Gamma (n+\beta)}{\Gamma (n+1)\Gamma (\beta)}b_j^n w^{n-1}dw\nonumber\\
                                                             &= a_j \sum_{n=1}^{\infty}\frac{\Gamma (n+\beta)}{\Gamma (n+1)\Gamma (\beta)}\frac{(b_j z)^{n}}{n}.\nonumber
\end{align}
We know that $\Gamma (n+\beta)/\Gamma (\beta)=(\beta)_n$, for $\beta\geq0$ and $n\geq0$, where $(\beta)_n$ denotes the shifted factorial defined by
$$
 (a)_n = a(a+1)\dots(a+n-1)  
$$
for $n>0$, and $(a)_0=1$ for a complex number $a$. Then
$$
\int_{0}^{z}\frac{\frac{a_{j}}{(1-b_{j} w)^\beta}-a_j}{w}dw  = a_j \sum_{n=1}^{\infty}\frac{(\beta)_n }{n}\frac{(b_j z)^{n}}{n!}.
$$
This implies that
$$
\left|\int_{0}^{z}\frac{\frac{a_{j}}{(1-b_{j} w)^\beta}-a_j}{w}\,dw \right| 
\le |a_j| \sum_{n=1}^{\infty}\frac{(\beta)_n }{n}\frac{|b_j|^{n}}{n!}
$$
on the circle of convergence $|z|=1$.
Now onward assume that $\beta>0$. We set $\delta=\beta/m$, $m\in\mathbb{N}$. For $|z|=1$, comparing the terms of above series with the corresponding terms of the convergent series
$$
 \sum_{n=1}^{\infty}\frac{1}{n^{1+\delta}}, 
$$
and to use the limit comparison test, we compute
\begin{align}
 \lim \limits_{n\rightarrow \infty}\Bigg|\frac{n^{\delta}(\beta)_n}{n!}\Bigg|&=\lim \limits_{n\rightarrow \infty}\Bigg|\frac{(\beta)_n}{(n-1)!n^\beta}\frac{(n-1)!n^{\delta+\beta}}{n!}\Bigg|.\nonumber
\end{align}
Since
$$
\frac{1}{\Gamma(\beta)} =\lim \limits_{n\rightarrow \infty}\frac{(\beta)_n}{(n-1)!n^\beta}.
$$
We obtain
$$
\lim \limits_{n\rightarrow \infty}\Bigg|\frac{n^{\delta}(\beta)_n}{n!}\Bigg|=\Bigg|\frac{1}{\Gamma(\beta)}\Bigg|\lim \limits_{n\rightarrow \infty}\Bigg|\frac{(n-1)!n^{\delta+\beta}}{n!}\Bigg|
                                                                            =\Bigg|\frac{1}{\Gamma(\beta)}\Bigg|\lim \limits_{n\rightarrow \infty}\Bigg|\frac{1}{n^{1-\delta-\beta}}\Bigg|,
$$
which tends to $0$ as $n\rightarrow\infty$, if $1-\delta-\beta>0$, i.e., if $\beta<m/(m+1)<1$, since $\delta=\beta/m$.
This implies that the series is absolutely convergent for $|z|=1$.
So $\psi_n(z)$ is a bounded analytic function in $\mathbb{D}$. 

Differentiating $\psi_n(z)$ with respect to $z$, we obtain
\begin{align*}
\psi_n^{'}(z)=&\frac{n}{g_\beta(0)}\Bigg(\frac{\sum_{j=1}^{k}\frac{a_j}{(1-b_{j} z)^\beta}+h(z)-\sum_{j=1}^{k}a_j-h(0)}{z}\Bigg)\psi_n(z)\nonumber\\
             =&\frac{n}{g_\beta(0)}\bigg(\sum_{j=1}^{k}\frac{\frac{a_{j}}{(1-b_{j} z)^\beta}-a_j}{z}+\frac{h(z)-h(0)}{z}\bigg)\psi_n(z).
\end{align*}
Since $\beta\leq\alpha$, $h(z)$ and $\psi_n(z)$ are bounded analytic functions in $\mathbb{D}$, it follows that
$$
 \sup\limits_{z \in \mathbb{D}}(1-|z|^2)^{\alpha}|\psi_n^{'}(z)|<\infty,
$$
as desired to have $\psi\in\mathcal{B}_\alpha$ and consequently $f\in\mathcal{B}_\alpha^0$ either for
$0<\beta\leq\alpha<1$ or $0<\beta<1<\alpha$ or $0<\beta<\alpha=1$.
\end{proof}
 \begin{remark}
  Define $\chi_n(z)=\psi_n(z)/\|\psi_n\|_{\mathcal{B}_{\alpha}}$, where $\psi_n$ is defined as in \eqref{1eq5.2} and \eqref{1eq5.10} according to the value of $\alpha$ and $\beta$.
  We know that
 \begin{align*}
  C_{g_\beta}(\chi_n(z))&= \frac{g_\beta(0)}{n} \chi_n(z)
 \end{align*}
 either for $0<\beta\leq\alpha<1$ or $\beta\leq1<\alpha$ or $\beta<\alpha=1$. Then we obtain
 \begin{align*}
 \|C_{g_\beta}(\chi_n)\|_{\mathcal{B}_{\alpha}}= \bigg|\frac{g_\beta(0)}{n}\bigg|.
 \end{align*}
 The right hand side approaches to $0$ as $n$ tends to $\infty$ i.e., $0$ is approximate eigenvalue of the generalized $\beta$-Ces\`{a}ro operator with the approximate eigenvector $\chi_n(z)$.
 \end{remark}

  By using the compactness properties (see Sections~\ref{1sec3} and \ref{1sec4}) of the $\beta$-Ces\`{a}ro operators from $\mathcal{B}_{\alpha}^0 $ to $\mathcal{B}_{\alpha}^0 $,
 for $0<\beta\leq\alpha<1$, $\beta\leq1<\alpha$ and $\beta<\alpha=1$, we establish the following theorem.
\begin{theorem} 
 The spectrum of the $\beta$-Ces\`{a}ro operator from $\mathcal{B}_{\alpha}^0 $ to $\mathcal{B}_{\alpha}^0 $,
 either $0<\beta\leq\alpha<1$ or $\beta\leq1<\alpha$ or $\beta<\alpha=1$, is
\begin{align*}
\sigma(C_{g_\beta}) &= \bigg \lbrace\frac{g_\beta(0)}{n},n \in \mathbb{N}\bigg \rbrace \bigcup \bigg \lbrace 0 \bigg \rbrace.
\end{align*}
\end{theorem}

\section{\bf An Application: separability of the space $\mathcal{B}_\alpha^0$}\label{1sec6}
Theorem \refeq{1theorem4.5} says that Ces\`{a}ro operator is a compact linear operator on $\mathcal{B}_{\alpha+1}^0$,
 for $\alpha>0$. Then by \cite[Theorem 8.2-3]{Kreyszig01}, the range $\mathfrak{R}(C_1)$ is separable.
 Therefore, we obtain the following property of $\mathcal{B}_\alpha^0$, with the help of the Ces\`{a}ro operator.
\begin{theorem}
  $\mathcal{B}_\alpha^0$ is a separable space in the space $\mathcal{B}_{\alpha+1}^0$, for $\alpha>0$.
\end{theorem}
\begin{proof}
 To prove this theorem, we need to show that the range $\mathfrak{R}(C_1)$ contains $\mathcal{B}_\alpha^0$, for $\alpha>0$, equivalently for $g\in\mathcal{B}_\alpha^0$ there exists an 
 $f\in\mathcal{B}_{\alpha+1}^0$ such that $C_1(f)=g$.
 
 Let $g\in\mathcal{B}_\alpha^0$ and define $f(z)=z(1-z)g'(z)$. Then $g'(z)=f(z)/z(1-z)$.
 Taking the line integral from $0$ to $z$, we get $g(z)=\int_0^z f(t)/t(1-t) dt$. Now we estimate
 \begin{align*}
  \sup \limits_{z\in\mathbb{D}}(1-|z|^2)^{\alpha+1}|f'(z)| =& \sup\limits_{z\in\mathbb{D}}(1-|z|^2)^{\alpha+1}|(1-z)g'(z)-zg'(z)\\
  &+z(1-z)g''(z)|\\
  \leq& 3\sup \limits_{z\in\mathbb{D}}(1-|z|^2)^{\alpha+1}|g'(z)|\\
  &+2 \sup 
  \limits_{z\in\mathbb{D}}(1-|z|^2)^{\alpha+1}|g''(z)|\\ 
  \leq& 3\sup \limits_{z\in\mathbb{D}}(1-|z|^2)^{\alpha}|g'(z)|\\
  &+2 \sup \limits_{z\in\mathbb{D}}(1-|z|^2)^{\alpha+1}|g''(z)|. 
 \end{align*}
By \cite[Proposition 8]{Zhu93}, we obtain
$$
 \sup \limits_{z\in\mathbb{D}}(1-|z|^2)^{\alpha+1}|g''(z)|<\infty,
$$
which says that $f\in\mathcal{B}_{\alpha+1}^0$.
\end{proof}

\bigskip
\noindent
{\bf Acknowledgement.} 
The second author would like to thank Prof. S. Ponnusamy for bringing the article \cite{AHLNP05} to his attention.


\begin{thebibliography}{99}
\bibitem{AHLNP05}
M. R. Agrawal, P. G. Howlett, S. K. Lucasa, S. Naik, and S. Ponnusamy,
{\em Boundedness of generalized Ces\`{a}ro averaging operators on
certain function spaces},
J. Math. Anal. Appl., {\bf 180} (2005), 333--344.

\bibitem{Albrecht05}  {E. Albrecht, T. L. Miller, M. M. Neumann,}
{\em Spectral properties of generalized Ces\`{a}ro operators on Hardy and weighted Bergman spaces},
 Arch. Math. (Basel), \textbf{85} (5) (2005), 446--459. 

\bibitem{Bela90}  {B. Bollobas,}
{\em Linear Analysis},
Cambridge University Press, 1990. 

\bibitem{DS93}
N. Danikas, A.G. Siskakis, 
{\em The Ces\`{a}ro operator on bounded analytic functions}, 
Analysis, {\bf 13} (1993), 195--199.

\bibitem{Duren83} {P. L. Duren,}
{\em Univalent Functions},
Springer-Verlag, New York, 1983.


\bibitem{Hartmann74} {F. W. Hartmann, and T. H. MacGregor,}
{\em Matrix transformations of univalent power series},
J. Aust. Math. Soc., \textbf{18} (1974), 419--435.

\bibitem{Hidetaka17} {H. Hidetaka,}
{\em Bloch-type spaces and extended Ces\`{a}ro operators in the unit ball of a complex Banach space},
Preprint ({\tt https://arxiv.org/abs/1710.11347})

\bibitem{Kreyszig01} {E. Kreyszig,}
{\em Introductory Functional Analysis with Application},
John Wiley \& Sons Inc., New York, 1989.

\bibitem{Limaye96} {B. V. Limaye,}
{\em Functional Analysis},
New Age International, New Delhi, 1996.

\bibitem{Mia92}
J. Miao, 
{\em The Ces\`{a}ro operator is bounded on $H^p$ for $0 <p< 1$}, 
Proc. Amer. Math. Soc., 
{\bf 116} (4) (1992), 1077--1079.

\bibitem{Miller90} {S. S. Miller, P. T. Mocanu,}
{\em Differential Subordinations: Theory and Applications},
Marcel Dekker, Inc., New York, 1990.

\bibitem{Ponnusamy18} {S. Ponnusamy, S. K. Sahoo, and T. Sugawa},
{\em Integral transforms of functions with bounded boundary rotation},
Preprint.

\bibitem{Julio10} { J. C. Ramos-Fernandez,}
{\em Composition operators on Bloch-Orlicz type spaces},
Appl. Math. Comput., 
\textbf{217} (7) (2010), 3392--3402.

\bibitem{Sis90} 
A. G. Siskakis, 
{\em The Ces\`{a}ro operator is bounded on $H^1$}, 
Proc. Amer. Math. Soc., {\bf 110} (4) (1990), 461--462.

\bibitem{Stevic06}  {S. Stevic,}
{\em Boundedness and Compactness of an integral operator on a weighted space on the polydisk},
Indian J. Pure Appl. Math., \textbf{37} (8) (2006), 343--355.

\bibitem{Stevic09}  {S. Stevic,}
{\em On a new integral-type operator from the Bloch space to Bloch-type spaces on the unit ball},
J. Math. Anal. Appl.,
\textbf{354} (2009), 426--434.

\bibitem{Stevic10}  {S. Stevic,}
{\em On an integral operator between Bloch-type spaces on the unit ball},
{Bul. Sci. Math., \textbf{134} (2010), 329--339}.

\bibitem{Xia97}
J. Xiao, 
{\em Ces\`{a}ro type operators on Hardy, BMOA and Bloch spaces}, 
Arch. Math., {\bf 68} (1997), 398--406.

\bibitem{Zhu93}  {K. Zhu,}
{\em Bloch type spaces of analytic functions},
{ Rocky Mountain J. Math., \textbf{23} (3) (1993), 1143--1177}.

\bibitem{Zhu05} {K. Zhu,}
{\em Spaces of Holomorphic Functions in the Unit Ball},
Springer, USA, 2005.

\bibitem{Zhu07} {K. Zhu,}
{\em Operator Theory in Function Spaces},
Amer. Math. Soc., USA, 2007.


\end{thebibliography}
\end{document}